%% file: emimm-arXiv.tex
\documentclass[aihp,preprint]{imsart}
\RequirePackage{amsthm,amsmath,amsfonts,amssymb}
\RequirePackage[numbers]{natbib}
\RequirePackage[colorlinks,citecolor=blue,urlcolor=blue]{hyperref}
\RequirePackage{graphicx}% uncomment this for including figures
\RequirePackage{mathrsfs}% for \mathscr used for diversities

\usepackage{xcolor} % only used in \QS and \MW
\usepackage{accents} % for \cev

\startlocaldefs
\theoremstyle{plain}
\newtheorem{lemma}{Lemma}[section]
\newtheorem{proposition}[lemma]{Proposition}
\newtheorem{corollary}[lemma]{Corollary}

\newtheorem{theorem}[lemma]{Theorem}

\theoremstyle{remark}
\newtheorem{definition}[lemma]{Definition}

\newcommand{\cev}[1]{\accentset{\leftharpoonup}{#1}}
\newcommand{\vecc}[1]{\accentset{\rightharpoonup}{#1}}

\newcommand{\bE}{\mathbb{E}}

\newcommand{\bN}{\mathbb{N}}
\newcommand{\bP}{\mathbb{P}}
\newcommand{\bR}{\mathbb{R}}

\newcommand{\cE}{\mathcal{E}}
\newcommand{\cF}{\mathcal{F}}

\newcommand{\cI}{\mathcal{I}}
\newcommand{\cJ}{\mathcal{J}}

\newcommand{\cN}{\mathcal{N}}

\newcommand{\sD}{\mathscr{D}}

\newcommand{\ff}{\mathbf{f}}

\newcommand{\fF}{\mathbf{F}}
\newcommand{\fN}{\mathbf{N}}

\newcommand{\ed}{\mbox{$ \ \stackrel{d}{=}$ }}

\newcommand{\besq}{{\tt BESQ}}

\newcommand{\skewer}{\ensuremath{\normalfont\textsc{skewer}}}
\newcommand{\skewerbar}{\ensuremath{\overline{\normalfont\textsc{skewer}}}}
\newcommand{\sskewer}{\ensuremath{\normalfont\textsc{sSkewer}}}
\newcommand{\clade}{\ensuremath{\normalfont\textsc{clade}}}

\def\Concat{ \mathop{ \raisebox{-2pt}{\Huge$\star$} } }

\def\concat{\star}

\newcommand{\odotip}{}
\newcommand{\pdip}{{\tt PDIP}}
\endlocaldefs

\begin{document}

\begin{frontmatter}

%%%%%%%%%%%%%%%%%%%%%%%%%%%%%%%%%%%%%%%%%%%%%%
%%                                          %%
%% Enter the title of your article here     %%
%%                                          %%
%%%%%%%%%%%%%%%%%%%%%%%%%%%%%%%%%%%%%%%%%%%%%%
\title{Two-sided immigration, emigration and symmetry properties\\ of self-similar interval partition evolutions}
%\title{A sample article title with some additional note\thanksref{T1}}
\runtitle{Interval partition evolutions with two-sided immigration}
%\thankstext{T1}{A sample of additional note to the title.}

\begin{aug}
		\author{\fnms{Quan} \snm{Shi}\thanksref{a}\ead[label=e1]{quanshi.math@gmail.com}}
\and
\author{\fnms{Matthias} \snm{Winkel}\thanksref{b}\ead[label=e2]{winkel@stats.ox.ac.uk}}
%%%%%%%%%%%%%%%%%%%%%%%%%%%%%%%%%%%%%%%%%%%%%%
%% Addresses                                %%
%%%%%%%%%%%%%%%%%%%%%%%%%%%%%%%%%%%%%%%%%%%%%%
\address[a]{Mathematical Institute, University of Mannheim, Mannheim D-68131, Germany. \printead{e1}}
\address[b]{Department of Statistics, University of Oxford, 24--29 St Giles', Oxford OX1 3LB, UK.  \printead{e2}}
\end{aug}

\begin{abstract}
 %Squared Bessel processes of dimension $\delta\in\bR$ are continuous-state branching processes with immigration when
%  $\delta>0$ or with emigration when $\delta<0$. 
Forman et al.\ (2020+) constructed $(\alpha,\theta)$-interval partition evolutions for 
$\alpha\in(0,1)$ and $\theta\ge 0$, in which the total sums of interval lengths (``total mass'') evolve as squared Bessel processes of dimension $2\theta$, where $\theta\ge 0$ acts as an immigration parameter. These evolutions have pseudo-stationary distributions related to regenerative Poisson--Dirichlet interval partitions. In this paper we study symmetry properties of $(\alpha,\theta)$-interval partition evolutions. Furthermore, we introduce a three-parameter family ${\rm SSIP}^{(\alpha)}(\theta_1,\theta_2)$ of self-similar interval partition evolutions that have separate left and right immigration parameters $\theta_1\ge 0$ and $\theta_2\ge 0$. They also have squared Bessel total mass processes of dimension $2\theta$, where $\theta=\theta_1+\theta_2-\alpha\ge-\alpha$ covers emigration as well as immigration.
Under the constraint $\max\{\theta_1,\theta_2\}\ge\alpha$, 
we prove that an ${\rm SSIP}^{(\alpha)}(\theta_1,\theta_2)$-evolution is pseudo-stationary for a new distribution on 
interval partitions, whose ranked sequence of lengths has Poisson--Dirichlet distribution with parameters $\alpha$ and 
$\theta$, but we are unable to cover all parameters without developing a limit theory for composition-valued Markov chains, which we do in a sequel paper.
\end{abstract}

%\begin{abstract}[language=french]
%???.
%\end{abstract}

\begin{keyword}[class=MSC2020]
\kwd{60J80} %branching process
%\kwd[Primary ]{60J80} 
\kwd{60J25} %Continuous-time Markov processes on general state spaces
\kwd{60G18} %Self-similar stochastic processes
%\kwd[; secondary ]{60J25, 60G18}
\end{keyword}

\begin{keyword}
\kwd{Poisson--Dirichlet distribution}
\kwd{interval partition}
%\kwd{Chinese restaurant process}
\kwd{branching with immigration, emigration}
\end{keyword}

\end{frontmatter}

%%%%%%%%%%%%%%%%%%%%%%%%%%%%%%%%%%%%%%%%%%%%%%
%%%% Main text entry area:

\section{Introduction.}

In this paper, we construct a three-parameter family of interval partition evolutions that generalises the two-parameter family recently introduced by Forman et al. \cite{IPPAT}. The two-parameter model was shown to induce evolutions, via suitable time-change and normalisation, that have Poisson--Dirichlet interval partitions ${\tt PDIP}^{(\alpha)}(\theta)$ as their stationary laws, $\alpha\in(0,1)$, $\theta\ge 0$. A related construction \cite{FVAT} yields measure-valued diffusions stationary with the two-parameter Pitman--Yor distribution \cite{IshwJame03,PitmYorPDAT,Teh06}. When projected onto ranked interval lengths (or atom sizes) \cite{Paper1-3}, these diffusions yield Petrov's \cite{Petrov09} Poisson--Dirichlet diffusions in the cases when $\theta\ge 0$. See also \cite{Ethier14,FengSun10,FengSun19,RuggWalk09,Ruggiero14}. Members of the two-parameter family were used in \cite{Paper4} to construct the Aldous diffusion that has the Brownian continuum random tree as its stationary distribution. See also \cite{Ald-web,LohrMytnWint18,NussWint20,Schweinsberg02}. The three-parameter family is relevant since it captures for each $\alpha\in(0,1)$ the full Poisson--Dirichlet parameter range $\theta>-\alpha$. This extended range is crucial for potential generalisations of the Aldous diffusion to continuum random trees that include multifurcating ones such as stable trees \cite{CFW,CuriKort2014,DuplMillShef14,DuquLeGall02,HM04,HPW,LeGaMier2011}.

For $M\ge 0$, an \emph{interval partition}  $\beta=\{U_i,i\in I\}$ of $[0,M]$ is a (countable) collection of disjoint open intervals $U_i=(a_i,b_i)\subseteq (0,M)$, 
such that the (compact) set of partition points 
$G(\beta):= [0,M]\setminus\bigcup_{i\in I}U_i$ has zero Lebesgue measure. 
We refer 
to the intervals $U\in\beta$ as \em blocks \em and to their lengths ${\rm Leb}(U)$ as their \em masses\em. We similarly refer to $\|\beta\|$ as the 
\em total mass \em of $\beta$, that is
$\|\beta\|:=\sum_{U\in\beta}{\rm Leb}(U)$. %=\sup\{\sup U,U\in\beta\}$.
We denote by $\cI_{H}$ the set of all 
interval partitions of $[0,M]$ for all $M\ge 0$. 
This space is equipped with the metric $d_H$, which is obtained by applying the Hausdorff metric to the sets of partition points: 
for every $\gamma,\gamma'\in \cI_{H}$, 
\begin{equation}\label{eq:dH}
d_{H} (\gamma, \gamma')
:= \inf \left\{r\ge 0\colon G(\gamma)\subseteq \bigcup_{x\in G(\gamma')} (x-r,x+r),~
G(\gamma')\subseteq \bigcup_{x\in G(\gamma)} (x-r,x+r) \right\}.
\end{equation}
The metric space $(\cI_H, d_H)$ is not complete, but it generates a Polish topology \cite{Paper1-0}. For $\beta\!\in\!\cI_H$, we write ${\rm rev}(\beta):=\{(\|\beta\|\!-\!b,\|\beta\|\!-\!a)\colon (a,b)\!\in\!\beta\}\in\cI_H$ for the left-right reversal of $\beta$. In \cite{GnedPitm05,PitmWink09}, a two-parameter 
family of interval partitions was introduced that places blocks of Poisson--Dirichlet masses ${\tt PD}^{(\alpha)}(\theta)$ into a regenerative 
random order. We denote their left-right reversals by ${\tt PDIP}^{(\alpha)}(\theta)$, $\alpha\in(0,1)$, $\theta\ge 0$. For $\theta=\alpha$, this is the distribution of the excursion intervals of a Bessel bridge \cite{PitmYorPDAT} with dimension parameter $2\alpha$. For $c>0$, we define a \emph{scaling
map} by\vspace{-0.1cm}
\[
c\odotip \beta:= \{(ca,cb)\colon (a,b)\in \beta\},\qquad\beta\in \cI_H.\vspace{-0.1cm}
\] 
Let $(\beta_a)_{a\in \mathcal{A}}$ be a family of interval partitions indexed by totally ordered set $(\mathcal{A}, \preceq)$. 
Let $S_{\beta}(a-):= \sum_{b\prec a} \|\beta_b\|$. 
We define the natural \emph{concatenation} of their blocks by\vspace{-0.1cm}
\[
\Concat_{a\in \mathcal{A}} \beta_a 
:= \big\{ 
\big(x+ S_{\beta}(a-), y + S_{\beta}(a-)\big) \colon (x,y)\in \beta_{a},\,  a\in \mathcal{A} 
\big\}.\vspace{-0.1cm}
\]
When $\mathcal{A} = \{1,2\}$, we denote this by $\beta_1 \concat \beta_2$. 

%By identifying any vector of positive integer $(n_1, \ldots, n_k)$ with the interval partition  
%\[
%\left\{\left(\sum_{j=1}^{i-1} n_j,\sum_{j=1}^{i} n_j\right), i=1,\ldots, k \right\}, 
%\]

Let us recall from \cite{IPPAT} the transition kernels of the two-parameter family. The kernels have the branching property (with immigration) 
under which each initial block of mass $b>0$ contributes independently to time $y$ with probability $1-e^{-b/2y}$. Specifically, for $r=1/2y>0$ 
and $b>0$, we consider independent   $G\sim{\tt Gamma}(\alpha,r)$, $\overline{\beta}\sim{\tt PDIP}^{(\alpha)}(\alpha)$, and a $(0,\infty)$-valued 
random variable $L_{b,r}^{(\alpha)}$ with Laplace transform\vspace{-0.1cm}
\begin{equation}
  \mathbb{E}\left[e^{-\lambda L_{b,r}^{(\alpha)}}\right] = \left(\frac{r+\lambda}{r}\right)^{\alpha}\frac{e^{br^2/(r+\lambda)}-1}{e^{br}-1}.
  \label{LMBintro}\vspace{-0.1cm}
\end{equation}
Then we define the distribution $\mu_{b,r}^{(\alpha)}$ of a random interval partition as \vspace{-0.1cm}
\begin{equation}
 \mu_{b,r}^{(\alpha)}=e^{-br}\delta_\emptyset+(1\!-\!e^{-br})\mathbb{P}\left(\{(0,L_{b,r}^{(\alpha)})\}\concat G\overline{\beta}\in\cdot\,\right).\vspace{-0.1cm}
\end{equation}

\begin{definition}[Transition kernel $\kappa^{\alpha,\theta}_y$]\label{def:kernel:sp}
 Fix $\alpha\in (0,1)$, $\theta\ge 0$ and let $\beta\in\cI_H$ and $y>0$. Then $\kappa^{\alpha,\theta}_y(\beta,\,\cdot\,)$ is defined to be the distribution of
$G^y\overline{\beta}_0\concat\Concat_{U\in\beta}\beta_U^y$ for independent $G^y\sim{\tt Gamma}(\theta,1/2y)$, $\overline{\beta}_0\sim{\tt PDIP}^{(\alpha)}(\theta)$, and $\beta_U^y\sim \mu_{{\rm Leb}(U),1/2y}^{(\alpha)}$, $U\in\beta$, where ``Leb'' denotes Lebesgue measure.\smallskip
\end{definition}

It was shown in \cite{IPPAT} that these kernels form the transition semigroup of a self-similar path-continuous Hunt process in $(\cI_H,d_H)$, which we call ${\rm SSIP}^{(\alpha)}(\theta)$-evolution. In Section 
\ref{sec:constr}, we recall a construction from spectrally positive stable L\'evy processes with jumps marked by squared Bessel excursions of dimension 
$-2\alpha$. This construction reveals that blocks evolve as independent ${\tt BESQ}(-2\alpha)$-processes, with further blocks 
created between existing blocks at a dense set of times, each evolving as a ${\tt BESQ}(-2\alpha)$-excursion. The 
contribution $G^y\overline{\beta}_0$ in the semigroup can be interpreted as ``immigration'' at rate $\theta\ge 0$, on the left-hand side. 
We note that the semigroup is left-right reversible for $\theta=\alpha$, which suggests that the right-most copy of $G\overline{\beta}$ may similarly be interpreted as immigration at rate $\alpha$. But with positive probability, no initial block contributes to time $y>0$, and modifying the two-parameter model  at the level of semigroups to include left-hand immigration at rate $\theta_1\ge 0$ and right-hand immigration at rate 
$\theta_2\ge 0$ is a challenge. We propose two approaches.

The starting point for the first approach is the following consequence of the symmetry properties of the semigroup of 
${\rm SSIP}^{(\alpha)}(\alpha)$-evolutions.

\begin{proposition}\label{prop:sym} Let $(\beta^y,\,y\ge 0)$ be an ${\rm SSIP}^{(\alpha)}(\alpha)$-evolution starting from $\beta\in\cI_H$. Then
  $({\rm rev}(\beta^y),\,y\ge 0)$ is an ${\rm SSIP}^{(\alpha)}(\alpha)$-evolution starting from ${\rm rev}(\beta)$. 
\end{proposition}

We construct a three-parameter family of ${\rm SSIP}^{(\alpha)}_{\dagger}(\theta_1,\theta_2)$-evolutions from ${\rm SSIP}^{(\alpha)}(\theta_1)$-evolutions, 
${\tt BESQ}(-2\alpha)$ processes and left-right-reversals of ${\rm SSIP}^{(\alpha)}(\theta_2)$-evolutions by repeatedly decomposing around a ``middle'' block, as follows.\smallskip

\begin{definition}\label{def:stopSSIP} Fix $\alpha\in(0,1)$ and $\theta_1,\theta_2\ge 0$. Let $\beta\in\cI_H$ and set $T_0:=0$, $\beta^0:=\beta$.  
  Inductively, for any $n\ge 0$, conditionally given $(\beta^y,\,0\le y\le T_n)$, proceed as follows.
  \begin{itemize}\item If $\beta^{T_n}=\emptyset$, set $T_i:=T_n$, $i\ge n+1$, and $\beta^y:=\emptyset$, $y\ge T_n$.\vspace{0.1cm}
    \item If $\beta^{T_n}\neq\emptyset$, denote by $U^{(n)}$ the longest interval in $\beta^{T_n}$, taking the leftmost of these if it is 
      not unique. Let $\beta_1^{(n)}:=\beta^{T_n}\!\cap\![0,\inf U^{(n)}]\!\in\!\cI_H$ be the partition to the left of $U^{(n)}$ and record the
      remainder to the right of $U^{(n)}$ in $\beta_2^{(n)}\!\in\!\cI_H$ such that 
      $\beta^{T_n}=\beta_1^{(n)}\concat\{(0,{\rm Leb}(U^{(n)}))\}\concat\beta_2^{(n)}$. For independent 
      ${\rm SSIP}^{(\alpha)}(\theta_j)$-evolutions $\gamma_j^{(n)}$ starting from $\beta_j^{(n)}$, $j=1,2$, and 
      $\mathbf{f}^{(n)}\sim{\tt BESQ}(-2\alpha)$ starting from ${\rm Leb}(U^{(n)})$ and absorbed at 
      $\zeta(\mathbf{f}^{(n)}):=\inf\{z\ge 0\colon\mathbf{f}^{(n)}(z)=0\}$, let
      \[T_{n+1}:=T_n+\zeta(\mathbf{f}^{(n)}),\qquad
         \beta^{T_n+s}:=\gamma_1^{(n)}(s)\concat\{(0,\mathbf{f}^{(n)}(s))\}\concat{\rm rev}(\gamma_2^{(n)}(s)),\quad 
         0\le s\le\zeta(\mathbf{f}^{(n)}).\]
  \end{itemize}
  If $T_n\uparrow T_\infty<\infty$, set $\beta^y:=\emptyset$, $y\ge T_\infty$. We refer to $(\beta^y,\,y\ge 0)$ as an
  ${\rm SSIP}^{(\alpha)}_{\dagger}(\theta_1,\theta_2)$-evolution, a \em self-similar interval partition evolution\em, to $\theta_1$ and $\theta_2$ as \em left and right immigration parameters\em, and to $(\|\beta^y\|,\,y\ge 0)$ as the \em total mass process.\em
\end{definition}

The subscript $\dagger$ in ${\rm SSIP}^{(\alpha)}_\dagger(\theta_1,\theta_2)$ acknowledges the fact that these processes are absorbed (``killed'') when they reach $\emptyset$. When $\theta>0$, this is not the case for ${\rm SSIP}^{(\alpha)}(\theta)$-evolutions. While almost surely neither process visits $\emptyset$ 
when $\theta\ge 1$, the state $\emptyset$ is an instantaneously reflecting boundary state of ${\rm SSIP}^{(\alpha)}(\theta)$-evolutions when 
$\theta\in(0,1)$. This relates to the well-known boundary behaviour of ${\tt BESQ}(2\theta)$, which we established in \cite[Theorem 1.4(iii)]{IPPAT} as the total mass evolution of ${\rm SSIP}^{(\alpha)}(\theta)$-evolutions. Specifically, recall that for any $m \ge 0$, $\delta \in \bR$, there is a unique strong solution of the equation
\[
Z_t = m +\delta t + 2 \int_0^t \sqrt{|Z_s|} d B_s,
\]
where $(B_t,t\ge 0)$ is a standard Brownian motion. We refer to \cite{GoinYor03} for general properties of such squared Bessel processes; let us here discuss the first hitting time of zero $\tau_0(Z):= \inf \{ t\ge 0\colon Z_t=0 \}$. Then $\tau_0(Z)$ is almost surely finite if and only if $\delta<2$. When $\delta<2$, the law of $\tau_0(Z)$ is described in (\cite[Equation (13)]{GoinYor03}) as the distribution of $m/2G$ with $G\sim \mathtt{Gamma}(1- \delta/2, 1)$. Furthermore, we define the \emph{lifetime of $Z$} by
\begin{equation}\label{eq:besq-zeta}
\zeta(Z):= \infty,	 ~\text{if}~ \delta>0,
\quad \text{and}\quad 
\zeta(Z):=\tau_0(Z),~\text{if}~ \delta\le 0.	
\end{equation}
We write $\besq_m(\delta)$ for the law of $Z:=(Z_{t\wedge\zeta(Z)}, t\ge 0)$, which we will refer to as a \em squared Bessel process starting from $m$ \em with dimension parameter $\delta$. When $\delta\le 0$, this process is absorbed at 0 at the end of its a.s.\ finite lifetime. We 
denote by $\besq_m^\dagger(\delta)$ the law of $(Z_{t\wedge\tau_0(Z)},t\ge 0)$, which differs from $\besq_m(\delta)$ only for $\delta\in(0,2)$.

We will check carefully that ${\rm SSIP}_\dagger^{(\alpha)}(\theta_1,\theta_2)$-evolutions are
well-defined as path-continuous processes in $(\cI_H,d_H)$ and establish the following analogue of \cite[Theorem 1.4(i)--(iii)]{IPPAT}. 

\begin{theorem} \label{thm:hunt} For each $\alpha\in(0,1)$, $\theta_1\ge 0$ and $\theta_2\ge 0$, an 
  ${\rm SSIP}^{(\alpha)}_\dagger(\theta_1,\theta_2)$-evolution $(\beta^y,y\ge 0)$ is a path-continuous Hunt process in $(\cI_H,d_H)$ with
  ${\tt BESQ}^\dagger(2\theta)$ total mass process, where $\theta=\theta_1+\theta_2-\alpha\ge-\alpha$. 
	It is self-similar with index $1$ in the sense that, for each $c>0$, $(c\beta^{y/c},y\ge 0)$ is an ${\rm SSIP}^{(\alpha)}_\dagger(\theta_1,\theta_2)$-evolution starting from $c\beta^0$. 
  Furthermore, if we stop an ${\rm SSIP}^{(\alpha)}(\theta_1)$-evolution at its first hitting time of $\emptyset$, we obtain an ${\rm SSIP}^{(\alpha)}_{\dagger}(\theta_1,\alpha)$-evolution. 
\end{theorem}

A key part of the proof is to show that ${\rm SSIP}^{(\alpha)}_\dagger(\theta_1,\theta_2)$-evolutions reach $\emptyset$ continuously. We will obtain the Markov property by applying Dynkin's criterion to the triple-valued process whose components are the mass of the ``middle'' block and the interval partitions on either side. This process inherits the Markov property from
${\rm SSIP}^{(\alpha)}(\theta_j)$-evolutions, $j=1,2$, and ${\tt BESQ}(-2\alpha)$ via standard results \cite{Mey75,Bec07} about suitably restarting Markov processes at stopping times.

%In an $\mathrm{SSIP}^{(\alpha)}_\dagger(\theta_1, \theta_2)$-evolution, informally speaking, each block evolves as a ${\tt BESQ}(-2\alpha)$, independently of the others; there is always immigration of rate $2\theta_1$ on the left, of rate $2\theta_2$ on the right, and immigration of rate $2\alpha$ between ``adjacent blocks''. Immigrating families are excursions growing from zero mass which evolve in the same way as other blocks.
%The caveat in this description is that there are no ``adjacent blocks'', typically any two blocks are separated by infinitely many other blocks (of summable masses).

The second approach to the three-parameter family starts from the representation \cite[(5.26)]{CSP} of ${\tt PD}^{(\alpha)}(\theta_1)$ as an 
$(\alpha,0)$-fragmentation of ${\tt PD}^{(0)}(\theta_1)$, in which every part of a ${\tt PD}^{(0)}(\theta_1)$-sequence is fragmented 
independently into ${\tt PD}^{(\alpha)}(0)$-proportions. See also \cite[Proposition 21]{PitmYorPDAT}. We noted in 
\cite[Proposition 3.6 and its proof]{IPPAT} that this has a refinement to interval partitions where each part in the ${\tt Beta}(\theta_1,1)$-stick-breaking scheme related to ${\tt PD}^{(0)}(\theta_1)$ is fragmented according to ${\tt PDIP}^{(\alpha)}(0)$. This decomposition arises 
naturally in the construction by marked L\'evy processes that we recall in Section \ref{sec:prel:clades} in that the ${\tt Beta}(\theta_1,1)$-proportions 
correspond to masses of those immigrant families that contribute to time $y$. These accumulate in a sequence of small contributions at the left 
end of the more recently immigrating families and each contribute a mass that further splits into ${\tt PDIP}^{(\alpha)}(0)$-proportions. Furthermore,
any ${\tt PDIP}^{(\alpha)}(0)$ consists of a leftmost block that is ${\tt Beta}(1-\alpha,\alpha)$-distributed, with the remainder being split into 
proportions according to ${\tt PDIP}^{(\alpha)}(\alpha)$.

In Section \ref{sec:large}, we make precise a coupling based on Poisson-thinning immigrant families, left-right-reversing scaled 
${\tt PDIP}^{(\alpha)}(\alpha)$, and regrouping blocks. For $\max\{\theta_1,\theta_2\}\ge\alpha$, this yields interval 
partition evolutions that we can use to construct ${\rm SSIP}^{(\alpha)}(\theta_1,\theta_2)$-evolutions which have $\emptyset$ as an entrance boundary when $\theta:=\theta_1+\theta_2-\alpha>0$, indeed as a reflecting boundary when $\theta\in(0,1)$, and they coincide with ${\rm SSIP}_\dagger^{(\alpha)}(\theta_1,\theta_2)$-evolutions 
(if stopped at the first hitting time of $\emptyset$). On the one hand, this yields the desired recurrent extension of ${\rm SSIP}^{(\alpha)}_\dagger(\theta_1,\theta_2)$-evolutions with unstopped ${\tt BESQ}(2\theta)$ total mass processes. 
%It also allows to extend ${\rm SSIP}^{(\alpha)}_\dagger(\theta_1,\theta_2)$ to enter from $\emptyset$ when $\theta\ge 1$. 
On the other hand, this approach lends itself
better to calculations than Definition \ref{def:stopSSIP}. Specifically, we extend \cite[Theorem 1.4(iv)]{IPPAT}, which states that starting an 
${\rm SSIP}^{(\alpha)}(\theta_1)$-evolution from an independently scaled multiple of a ${\tt PDIP}^{(\alpha)}(\theta_1)$, the marginal 
distribution at all times is an independently scaled multiple of a ${\tt PDIP}^{(\alpha)}(\theta_1)$. In view of this property, we refer to 
${\tt PDIP}^{(\alpha)}(\theta_1)$ as a \em pseudo-stationary distribution \em of ${\rm SSIP}^{(\alpha)}(\theta_1)$. 

\begin{theorem} \label{thm:pseudo}
  Let $\theta_1\ge \alpha$ and $\theta_2\ge 0$. For independent $B\sim{\tt Beta}(\theta_1-\alpha,\theta_2)$, $B^\prime\sim{\tt Beta}(1-\alpha,\theta_1)$ and $\overline{\beta}_j\sim{\tt PDIP}^{(\alpha)}(\theta_j)$, $j=1,2$, the distribution of 
  \begin{equation}B(1-B^\prime)\overline{\beta}_1\concat\{(0,BB^\prime)\}\concat(1-B){\rm rev}(\overline{\beta}_2)\label{eq:pseudo}
  \end{equation}
  is a pseudo-stationary distribution of the ${\rm SSIP}^{(\alpha)}(\theta_1,\theta_2)$-evolutions. 
\end{theorem}

The two approaches and associated three-parameter models allow us to study further related processes and properties that generalise 
straightforwardly from corresponding results for the two-parameter family of \cite{IPPAT}, but they also leave several questions open, which 
merit further exploration. 
\begin{itemize} \item The self-similarity of the construction allows us to de-Poissonize in the sense that we can time-change an 
  ${\rm SSIP}^{(\alpha)}_\dagger(\theta_1,\theta_2)$- or equivalently an ${\rm SSIP}^{(\alpha)}(\theta_1,\theta_2)$-evolution 
  $(\beta^y,\,y\ge 0)$ by the time-change $\tau(u):= \inf \left\{ y\ge 0\colon \int_0^y \|\beta^z\|^{-1} d z>u \right\}$, $u\ge 0$, and 
  normalise to unit total mass $\beta^{\tau(u)}/\|\beta^{\tau(u)}\|$, $u\ge 0$. This yields a Hunt process, which, in the context of 
  Theorem \ref{thm:pseudo}, will have the distribution of \eqref{eq:pseudo} as a stationary distribution.\vspace{0.1cm}
\item ${\tt PD}^{(\alpha)}(\theta)$ is well-known to have a non-trivial $\alpha$-diversity. This was strengthened to uniform ``local time'' 
  approximations associated with a single ${\tt PDIP}^{(\alpha)}(\theta)$ in \cite[Proposition 6(iv)]{PitmWink09} and with constructions from all levels
  in marked L\'evy processes in \cite[Theorem 1]{Paper0}. With a bit of work to control diversities when approaching $\emptyset$,
  it can be shown that ${\rm SSIP}^{(\alpha)}_\dagger(\theta_1,\theta_2)$- and ${\rm SSIP}^{(\alpha)}(\theta_1,\theta_2)$-evolutions have 
  continuously evolving diversity processes $y\mapsto\sD_{\beta^y}$, where notation means
  $\sD_{\beta} (t):= \Gamma(1-\alpha) \lim_{h\downarrow 0} h^{\alpha}\# \{(a,b)\in \beta :~|b-a|>h , b\le t\}$, $t\ge 0$, (if this limit exists).\vspace{0.1cm}
\item The second approach is subject to the restriction $\theta_1\ge\alpha$, or $\max\{\theta_1,\theta_2\}\ge\alpha$ under the 
  exploitation of symmetry properties. Recall $\theta:=\theta_1+\theta_2-\alpha$. If we distinguish according to the absorbing, reflecting and transient boundary behaviours of the 
  ${\tt BESQ}(2\theta)$ total mass process, when $\theta\le 0$, $\theta\in(0,1)$ and $\theta\ge 1$, respectively, this restriction excludes all
  absorbing cases and some cases in the reflecting regime, but the entire transient regime is already covered. We will address the remaining cases 
  in a sequel paper \cite{QSMW2}, where we take a third approach to the three-parameter family. This involves the development of a limit 
  theory for composition-valued Markov chains, which is of independent interest, and which in the two-parameter special case resolves a 
  conjecture of \cite{RogeWink20}. See also \cite{RivRiz} for a different approach to convergence results in the corresponding de-Poissonized 
  setting, enhancing \cite{Petrov09}. \vspace{0.1cm}
\item We will show in \cite{Paper1-3} that de-Poissonized ${\rm SSIP}^{(\alpha)}(\theta_1)$-, as well as 
  ${\rm SSIP}^{(\alpha)}(\theta_1,\theta_2)$- and ${\rm SSIP}^{(\alpha)}(\theta_1,\theta_2)$-evolutions, are such that the associated process 
  of projections onto ranked block sizes are Petrov's ${\tt PD}^{(\alpha)}(\theta)$-diffusions, where we note that 
  ${\rm SSIP}^{(\alpha)}(\theta_1)$-evolutions cover only $\theta\ge 0$, while ${\rm SSIP}^{(\alpha)}_\dagger(\theta_1,\theta_2)$-evolutions here cover Petrov's full parameter range $\theta>-\alpha$, as $\theta:=\theta_1+\theta_2-\alpha\ge-\alpha$ when $\alpha\in(0,1)$, as well as the boundary case $\theta=-\alpha$ of processes that degenerate by absorption in 
  $\{(0,1)\}\in\cI_H$ or $(1,0,0,\ldots)$, respectively.\vspace{0.1cm}
\item The pseudo-stationary distribution of Theorem \ref{thm:pseudo} does not align with the standard examples of Beta-Gamma algebra in which parameters satisfy certain addition rules, under which products of independent Beta variables give rise to further Beta or Dirichlet variables. Here, the
  first parameter $\theta_1-\alpha$ of $B$ represents a component that gets further split according to a Beta variable whose parameters sum to
  $\theta_1+1-\alpha$, which is 1 too high -- and it seems that there is no apparent size-biased selection that would naturally have such effects, see e.g. \cite{CSP}. 
  We will revisit this intriguing decomposition of the pseudo-stationary distribution around a ``middle'' block
  in the sequel paper \cite{QSMW2}, where we identify several similar representations as decompositions around different ``middle'' blocks.  
\end{itemize}

\subsection{Organisation of the paper.}
The structure of this paper is as follows. We first recall in Section~\ref{sec:pre} the topology \cite{Paper1-0} and main examples 
\cite{GnedPitm05,PitmWink09} of interval partitions, and 
we slightly develop results from \cite{Paper1-1,Paper1-2}, here discussing symmetry properties of ${\rm SSIP}^{(\alpha)}(\alpha)$-evolutions including a short proof of 
Proposition \ref{prop:sym}.  
In Section~\ref{sec:ip}, we introduce triple-valued $\mathrm{SSIP}_\dagger^{(\alpha)}(\theta_1,\theta_2)$-evolutions, study their properties, and prove Theorem~\ref{thm:hunt}. 
In Section~\ref{sec:prel:clades}, we recall from \cite{Paper1-1,Paper1-2, IPPAT} the construction of ${\rm SSIP}^{(\alpha)}(\theta)$-evolutions 
from marked stable L\'evy processes and Poisson random measures. In Section~\ref{sec:large}, we make precise the construction of 
$\mathrm{SSIP}^{(\alpha)}(\theta_1,\theta_2)$-evolutions when $\theta_1\ge\alpha$ and prove Theorem \ref{thm:pseudo}. 
The appendix contains proofs of two key technical lemmas needed in Section \ref{sec:ip}.

\section{Preliminaries on the transition description of the two--parameter family ${\rm SSIP}^{(\alpha)}(\theta)$, $\alpha\in(0,1)$, $\theta\ge 0$.}\label{sec:pre}

Throughout this paper, we fix a parameter $\alpha\in (0,1)$. In this section, we recall and develop some properties of 
${\rm SSIP}^{(\alpha)}(\theta)$-evolutions. Specifically, Section \ref{sec:top} briefly revisits the topology on $\cI_H$ of \cite{Paper1-0}. In Section \ref{sec:PDIP}, we discuss the two-parameter family ${\tt PDIP}^{(\alpha)}(\theta)$ of interval partitions \cite{GnedPitm05,PitmWink09} that arise as stationary distributions and in transition kernels. In Section \ref{sec:sym}, we discuss symmetry properties and include a short proof of Proposition \ref{prop:sym}.

\subsection{The topology generated by the metric space $(\cI_H,d_H)$.}\label{sec:top}
Recall from the introduction that $\cI_{H}$ is the space of interval partitions endowed with the metric $d_H$ of \eqref{eq:dH}. 
We next endow $\cI_H$ with another metric $d_H^\prime$ introduced in \cite{Paper1-0}. 
Let $[n] := \{1,2,\ldots,n\}$.
For $\beta,\gamma\in \cI_H$, a \emph{correspondence} from $\beta$ to $\gamma$ is a finite sequence of ordered pairs of intervals $(U_1,V_1),\ldots,(U_n,V_n) \in \beta\times\gamma$, $n\geq 0$, where the sequences $(U_j)_{j\in [n]}$ and $(V_j)_{j\in [n]}$ are each strictly increasing in the left-to-right ordering of the interval partitions.
The Hausdorff \emph{distortion} of a correspondence $(U_j,V_j)_{j\in [n]}$ from $\beta$ to $\gamma$, denoted by $\mathrm{dis}_H(\beta,\gamma,(U_j,V_j)_{j\in [n]})$, is defined to be the maximum of the following two quantities:
\begin{enumerate}%[label=(\roman*), ref=(\roman*)]
	\item $\sum_{j\in [n]}|\mathrm{Leb}(U_j)-\mathrm{Leb}(V_j)| + \|\beta\|- \sum_{j\in [n]}\mathrm{Leb}(U_j)$, \label{item:IP_m:mass_1}\vspace{0.2cm}
	\item $\sum_{j\in [n]}|\mathrm{Leb}(U_j)-\mathrm{Leb}(V_j)| + \|\gamma\| - \sum_{j\in [n]}\mathrm{Leb}(V_j)$, \label{item:IP_m:mass_2}
%	\item $\sup_{j\in [n]}|\sD _{\beta}(U_j) - \sD _{\gamma}(V_j)|$,
%	\item $|\sD _{\beta}(\infty) - \sD _{\gamma}(\infty)|$.
\end{enumerate}
% Similarly, the \emph{Hausdorff distortion} of a correspondence $(U_j,V_j)_{j\in [n]}$ between $\beta,\gamma\in\cI_H$, denoted by $\mathrm{dis}_H(\beta,\gamma,(U_j,V_j)_{j\in[n]})$, is defined to be the maximum of (i)--(ii).
For $\beta,\gamma\in\cI_H$ we define
\begin{equation}\label{eq:IPH:metric_def}
  d_{H}^\prime(\beta,\gamma) := \inf_{n\ge 0,\,(U_j,V_j)_{j\in [n]}}\mathrm{dis}_H\big(\beta,\gamma,(U_j,V_j)_{j\in [n]}\big),
\end{equation}
%For $\beta,\gamma\in\cI_{\alpha}$ we define
%\begin{equation}\label{eq:IP:metric_def}
%d_{\alpha}(\beta,\gamma) := \inf_{n\ge 0,\,(U_j,V_j)_{j\in [n]}}\mathrm{dis}_\alpha\big(\beta,\gamma,(U_j,V_j)_{j\in [n]}\big),
%\end{equation}
where the infimum is over all correspondences from $\beta$ to $\gamma$. 
%\end{definition}

\begin{lemma}[Theorems 2.3-2.4 of \cite{Paper1-0}]\label{lem:cI}
	The metric spaces $(\cI_H,d_H)$ and $(\cI_H,d_H^\prime)$ generate the same separable topology. The space $(\cI_H,d_H^\prime)$ is 
	complete, while $(\cI_H,d_H)$ is not complete. In particular, the topology is Polish.
\end{lemma}

%From now on we fix $\alpha\in (0,1)$ and write $(\cI,d_{\cI}):= (\cI_{\alpha}, d_{\alpha})$. 

\subsection{Poisson--Dirichlet interval partitions ${\tt PDIP}^{(\alpha)}(\theta)$.}\label{sec:PDIP}

\begin{definition}[$\mathtt{PDIP}^{(\alpha)}( \theta)$]\label{defn:pdip-alphatheta}
	Fix $\alpha\in (0,1)$ and $\theta\ge 0$. Let $(Z_{\alpha,\theta}(t),\,t\ge0)$ denote a subordinator with Laplace exponent
	\begin{equation*}
	\begin{split}
	\Phi_{\alpha,\theta}(q) := \frac{q\Gamma(q+\theta)\Gamma(1-\alpha)}{\Gamma(q+\theta+1-\alpha)} \quad \text{if }\theta>0, 
	~\text{or} \quad \Phi_{\alpha,0}(q) := \frac{\Gamma(q+1)\Gamma(1-\alpha)}{\Gamma(q+1-\alpha)}, \qquad q\ge 0. 
	\end{split}
	\end{equation*}
	Let $(Z_{\alpha,\theta}(t-),\,t\ge 0)$ denote the left-continuous version of this subordinator. 
	We write $\mathtt{PDIP}^{(\alpha)}( \theta)$ to denote the law of a \emph{Poisson-Dirichlet $(\alpha,\theta)$ interval partition}, which is a random interval partition distributed like
	$$\left\{ \left(e^{-Z_{\alpha,\theta}(t)},e^{-Z_{\alpha,\theta}(t-)}\right)\!  \colon  t\ge 0,\ Z_{\alpha,\theta}(t-)\neq Z_{\alpha,\theta}(t)\right\}.$$
\end{definition}
A Poisson--Dirichlet $(\alpha, \theta)$ interval partition is the reversal of a \emph{regenerative $(\alpha,\theta)$ interval partition} studied in \cite{GnedPitm05} and \cite{PitmWink09}.  
It also describes the limiting proportions of customers at tables in the
ordered Chinese restaurant process of \cite{PitmWink09}. The special case ${\tt PDIP}^{(\alpha)}(0)$ can be obtained from an $\alpha$-stable 
subordinator $(\sigma(t),t\ge 0)$ as the interval partition of $[0,1]$ obtained from the complement of the range of $1-\sigma(t)$, $t\ge 0$, 
restricted to $[0,1]$, or equivalently as the left-right-reversal of the interval partition formed by the excursion intervals in [0,1] of a (squared) Bessel process of dimension $2-2\alpha$, including the incomplete excursion stopped at time 1. In \cite{IPPAT} we noted the following alternative representation, which builds the general ${\tt PDIP}^{(\alpha)}(\theta)$ from ${\tt PDIP}^{(\alpha)}(0)$, refining the 
${\tt PD}^{(\alpha)}(\theta)$ analogue of \cite[(5.26)]{CSP} and \cite[Proposition 21]{PitmYorPDAT}.
\begin{lemma}[Proposition 3.6 and its proof in \cite{IPPAT}] Let $B_i\sim{\tt Beta}(\theta,1)$, $\overline{\beta}_i\sim{\tt PDIP}^{(\alpha)}(0)$, $i\ge 1$, be independent. Then 
  \[\Concat_{k=\infty}^1(1-B_k)\left(\prod_{i=1}^{k-1}B_i\right)\overline{\beta}_k\sim{\tt PDIP}^{(\alpha)}(\theta),\]
  where the indexation of the concatenation operator means that the $(k+1)$st term is placed to the left of the $k$th, $k\ge 1$. 
\end{lemma}

We also record here from \cite[Proposition~2.2(iv)]{Paper1-2} a decomposition for easier reference: with independent $B\sim \mathtt{Beta}(\alpha, 1-\alpha)$ and $\bar\beta\sim \mathtt{PDIP}^{(\alpha)}(\alpha)$, we have 
\begin{equation}\label{eq:pdip:0-alpha}
\{(0, 1-B)\}\concat B \bar{\beta} \sim \mathtt{PDIP}^{(\alpha)}(0). 
\end{equation}

\subsection{${\rm SSIP}^{(\alpha)}(\theta)$-evolutions and their left-right reversals.}\label{sec:sym}

Recall from the introduction the definition of ${\rm SSIP}^{(\alpha)}(\theta)$-evolutions via their transition kernels and the terminology
``total mass process'' for the process $(\|\beta^y\|,y\ge 0)$ associated with any interval partition evolution $(\beta^y,y\ge 0)$. In this section 
we define left-right-reversed ${\rm SSIP}^{(\alpha)}(\theta)$-evolutions and also provide an (elementary) proof of 
Proposition \ref{prop:sym}.
Let $\beta\in\cI_H$ and recall that we denote its left-right reversal by 
\[
{\rm rev}(\beta):=\{ (\|\beta\|-b, \|\beta\|-a):~ (a,b)\in\beta\}.
\]

%\pagebreak

\begin{definition}\label{defn:rssip}
	Let $\alpha\in (0,1)$, $\theta\ge 0$, and $\beta_0\in \cI_H$. Let $(\beta^y,y\ge 0)$ be an $\mathrm{SSIP}^{(\alpha)}(\theta)$-evolution 
	starting from $\mathrm{rev}(\beta_0)$. Then we call the process $({\rm rev}(\beta^y),y\ge 0)$ a \emph{(left-right-)reversed self-similar 
	interval partition evolution with parameters $\alpha$ and $\theta$}, abbreviated as \emph{ $\mathrm{RSSIP}^{(\alpha)}(\theta)$-evolution, starting from $\beta_0$}. 
\end{definition}

\begin{proposition}\label{prop:Hunt+totalmass} ${\rm SSIP}^{(\alpha)}(\theta)$-evolutions and ${\rm RSSIP}^{(\alpha)}(\theta)$-evolutions are path-continuous Hunt processes. Their total mass processes are ${\tt BESQ}(2\theta)$-processes.
\end{proposition}
\begin{proof} The claims for ${\rm RSSIP}^{(\alpha)}(\theta)$-evolutions follow from the corresponding result for ${\rm SSIP}^{(\alpha)}(\theta)$-evolutions, \cite[Theorem 1.4]{IPPAT}, since left-right-reversal ${\rm rev}\colon (\mathcal{I}_H,d_H)\rightarrow (\mathcal{I}_H,d_H)$ is a total-mass-preserving homeomorphism.
\end{proof}

\begin{lemma}[Corollary~10.2 of \cite{GnedPitm05}]\label{lem:rev-pdip}
	Let $\beta\sim \mathtt{PDIP}^{(\alpha)}( \alpha)$. Then $\mathrm{rev}(\beta)\sim \mathtt{PDIP}^{(\alpha)}( \alpha)$.
\end{lemma}
%\begin{proof}
%	We can describe $\mathtt{PDIP}^{(\alpha)}( \alpha)$ as the law of the interval partition induced by the zero points of a Bessel $(2-2\alpha)$ bridge from zero to zero $X:= (X_t, t\in [0,1])$. The claim then follows from the fact that the process $(X_{1-t}, t\in [0,1])$ has the same distribution as $X$; see e.g.\ \cite[Exercise XI 3.7]{RevuzYor}. 
%\end{proof}

Let $(\beta^y,y\ge 0)$ be an $\mathrm{SSIP}^{(\alpha)}(\alpha)$-evolution starting from any $\beta_0\in\cI_H$. Recall that Proposition \ref{prop:sym} claims that left-right-reversing $\beta^y$, $y\ge 0$, yields another ${\rm SSIP}^{(\alpha)}(\alpha)$-evolution. 

\begin{proof}[Proof of Proposition \ref{prop:sym}]
        By Definition \ref{def:kernel:sp} and Lemma \ref{lem:rev-pdip}, the marginal distribution of ${\rm rev}(\beta^y)$ is, as required for an 
        ${\rm SSIP}^{(\alpha)}(\alpha)$-evolution starting from ${\rm rev}(\beta_0)$.	As the initial state $\beta_0$ was arbitrary and because of the Markov property and the 
        identity of the marginal distributions, we identify the finite-dimensional distributions. 
        Finally, we note that $({\rm rev}(\beta^y),\,y\ge 0)$ is also path-continuous, and this completes the identification as an ${\rm SSIP}^{(\alpha)}(\alpha)$-evolution. 
\end{proof}
\begin{corollary}\label{cor:sym}
 An $\mathrm{RSSIP}^{(\alpha)}(\alpha)$-evolution is an $\mathrm{SSIP}^{(\alpha)}(\alpha)$-evolution.
\end{corollary}
\begin{proof}
For any $\beta_0\in \cI_H$, let $\beta:=(\beta^y,y\ge 0)$ be an $\mathrm{SSIP}^{(\alpha)}(\alpha)$-evolution starting from ${\rm rev}(\beta_0)$. 
By definition, $({\rm rev}(\beta^y),y\ge 0)$ is a $\mathrm{RSSIP}^{(\alpha)}(\alpha)$-evolution starting from $\beta_0$. 
At the same time, it follows from Proposition~\ref{prop:sym} that $({\rm rev}(\beta^y),y\ge 0)$ 
	is an $\mathrm{SSIP}^{(\alpha)}(\alpha)$-evolution starting from $\beta_0$. 
\end{proof}

We also recall the following relationship between ${\rm SSIP}^{(\alpha)}(0)$-evolutions and ${\rm SSIP}^{(\alpha)}(\alpha)$-evolutions, which we will apply as it stands and in combination with Definition~\ref{defn:rssip} as a relationship between $\mathrm{RSSIP}^{(\alpha)}(0)$-evolutions and ${\rm RSSIP}^{(\alpha)}(\alpha)$-evolutions. 

\begin{lemma}[Proposition 3.15 of \cite{Paper1-1}]\label{lem:split} For $m>0$ and $\gamma\in\mathcal{I}_H$, consider on the one hand an 
  ${\rm SSIP}^{(\alpha)}(0)$-evolution $(\widetilde{\beta}^y,y\ge 0)$ starting from $\{(0,m)\}\concat\gamma$, and on the one hand independent 
  $\mathbf{f}\sim{\tt BESQ}_m(-2\alpha)$ and an ${\rm SSIP}^{(\alpha)}(\alpha)$-evolution $(\beta^y,y\ge 0)$ starting from $\gamma$. 
  Denote by $Y$ the infimum of jump times of the leftmost block of $(\widetilde{\beta}^y,y\ge 0)$. Then
  $$\Big(\widetilde{\beta}^y,\,y\in[0,Y)\Big)\overset{d}{=}\Big(\{(0,\mathbf{f}(y))\}\concat\beta^y,\,y\in[0,\zeta(\mathbf{f}))\Big).$$
\end{lemma}

Finally, we recall the following consequence of the form of the semi-groups of ${\rm SSIP}^{(\alpha)}(\theta)$-evolutions.

\begin{proposition}\label{prop:concat} Consider an independent pair consisting of an ${\rm SSIP}^{(\alpha)}(\theta)$-evolution 
  $(\beta_1^y,\,y\ge 0)$ starting from $\beta_1^0\in\mathcal{I}_H$ and an ${\rm SSIP}^{(\alpha)}(0)$-evolution $(\beta_2^y,\,y\ge 0)$ 
  starting from $\beta_2^0\in\mathcal{I}_H$. Then $(\beta_1^y\concat\beta_2^y,\,y\ge 0)$ is an ${\rm SSIP}^{(\alpha)}(\theta)$-evolution 	
  starting from $\beta_1^0\concat\beta_2^0$. 
\end{proposition}

\section{$\mathrm{SSIP}_\dagger^{(\alpha)}(\theta_1,\theta_2)$-evolutions for $\theta_1,\theta_2\ge 0$, and the proof of Theorem \ref{thm:hunt}.}\label{sec:ip}
  
%-------------------------------------------------------------------------------
%-------------------------------------------------------------------------------
\subsection{Triple-valued $\mathrm{SSIP}_{\!\dagger}^{(\alpha)} (\theta_1, \theta_2)$-evolutions.}
Fix $\alpha\in (0,1)$. 
Let $\cJ:=\big(\cI_H \times  (0,\infty)\times \cI_H\big) \cup \{(\emptyset,0,\emptyset)\} $ and equip $\cJ$ with the metric 
\[
d_{\cJ} (( \beta_1,m, \beta_2), ( \beta'_1, m',\beta'_2)) := d_{H}(\beta_1,\beta'_1)+|m\!-\!m'|+ d_{H}(\beta_2,\beta'_2).
\] The space $(\cJ, d_{\cJ})$ is a Borel subset of a Polish space, since $(\cI_H, d_{H})$ is a Polish space by Lemma~\ref{lem:cI}.  

We define a function $\phi: ~ \cI_H \to \cJ$, as follows. Let $\beta\in \cI_H$. For the purpose of defining $\phi(\beta)$, let $U$ be the longest interval in $\beta$; if the longest interval is not unique, then we take $U$ to be the leftmost longest interval. Then we set
\begin{equation}\label{eq:phi}
\phi(\beta) := ( \beta \cap (0,\inf U), \mathrm{Leb}(U), \beta \cap (\sup U, \|\beta\|) - \sup U ). 
\end{equation}
By convention, $\phi(\emptyset):= (\emptyset,0,\emptyset)$. 

\begin{definition}[Triple-valued $\mathrm{SSIP}_{\dagger}^{(\alpha)}(\theta_1,\theta_2)$-evolution] \label{defn:ipe}
	Let $\theta_1 ,\theta_2 \ge 0$ and $(\beta_1^0,m^0,\beta_2^0)\in \cJ$. 
	We define a \emph{$\cJ$-valued $\mathrm{SSIP}_{\!\dagger}^{(\alpha)}( \theta_1, \theta_2)$-evolution} $(( \beta_1^y,m^y, \beta_2^y), y\ge 0)$ starting from $(\beta_1^0,m^0,\beta_2^0)$ by the following construction. 
	
	Set $T_0 := 0$. 
	For $n\ge 0$, suppose by induction that we have constructed the process for the time interval $[0,T_n]$. 
	\begin{itemize}%[leftmargin=.7cm]
		\item If $( \beta_1^{T_{n}},m^{T_{n}},  \beta_2^{T_{n}}) \!=\!(\emptyset, 0,\emptyset)$, then we set 
		$T_{i}:= T_n$ for every $i\ge n+1$, and $( \beta_1^y,m^y, \beta_2^y) := (\emptyset,0, \emptyset)$, $y\ge T_n$.
		\item If $( \beta_1^{T_{n}},m^{T_{n}},  \beta_2^{T_{n}}) \!\ne\! (\emptyset, 0,\emptyset)$, then conditionally on the history, consider, independently, an $\mathrm{SSIP}^{(\alpha)}( \theta_1)$-evolution $\gamma^{(n)}_1$ starting from $\beta_1^{T_n}$,   an  $\mathrm{RSSIP}^{(\alpha)}(\theta_2)$-evolution  $\gamma^{(n)}_2$ starting from $\beta_2^{T_n}$, and $\ff^{(n)}\sim \besq_{m^{T_n}}(-2 \alpha)$. Set
		\[T_{n+1}:= T_n + \zeta(\ff^{(n)}),\qquad
		\left( \beta_1^{T_n+y}, m^{T_n+y}, \beta_2^{T_n+y}\right) := \left( \gamma^{(n)}_1(y) ,\ff^{(n)}(y) ,\gamma^{(n)}_2(y) \right), \qquad 0\le y<\zeta(\ff^{(n)}). 
		\]
		Furthermore, with $\phi$ the function defined in \eqref{eq:phi}, we set
		\[
		\left( \beta_1^{T_{n+1}}, m^{T_{n+1}}, \beta_2^{T_{n+1}}\right)
		:=%\phi(\beta^{T_{n+1}-} ) = 
		\phi\left(\beta_1^{T_{n+1}-} \concat \beta_2^{T_{n+1}-}\right).	
		\]
	\end{itemize}
	We refer to $T_n$, $n\ge1$, as the \emph{renaissance times} and $T_{\infty}:= \sup_{n\ge 1} T_n \in [0, \infty]$ as the \emph{degeneration time}. If $T_{\infty}< \infty$, then by convention we set  
	$( \beta_1^y, m^y,\beta_2^y) :=( \emptyset,0, \emptyset)$ for all $y\ge  T_{\infty}$. 
\end{definition}

By construction, the process $(\beta^y:=\beta_1^y \concat \{(0,m^y)\}\concat\beta_2^y,\, y\ge 0)$ satisfies the Definition \ref{def:stopSSIP} of an
\emph{$\cI_H$-valued $\mathrm{SSIP}_{\!\dagger}^{(\alpha)}( \theta_1, \theta_2)$-evolution} starting from $\beta^0:=\beta_1^0 \concat \{(0,m^0)\}\concat\beta_2^0$.

The following observation is a direct consequence of the construction. 

\begin{proposition}[Left-right reversal]
	For $\theta_1,\theta_2 \ge 0$, let  $((\beta_1^y, m^y,  \beta_2^y), y\ge 0)$ be a $\cJ$-valued $\mathrm{SSIP}_{\!\dagger}^{(\alpha)}( \theta_1, \theta_2)$-evolution and $(\beta^y, y\ge 0)$ its associated $\cI_H$-valued process. 
	Then  $\big((\mathrm{rev}(\beta_2^y), m^y,  \mathrm{rev}(\beta_1^y)), y\ge 0\big)$ is a $\cJ$-valued $\mathrm{SSIP}_{\!\dagger}^{(\alpha)}( \theta_2, \theta_1)$-evolution, and its associated $\cI_H$-valued process is $(\mathrm{rev} (\beta^y), y\ge 0)$. 
\end{proposition}

\begin{proof}
	By the construction in Definition~\ref{defn:ipe}, we have the identity, for every $n\ge 0$ and $y\in [T_n, T_{n+1})$,
	\[
	\left(\mathrm{rev}(\beta_2^y), m^y, \mathrm{rev}(\beta_1^y)\right) = \left( \mathrm{rev}\big(\gamma^{(n)}_2(y\!-\! T_n)\big) ,\ff^{(n)}(y\!-\!T_n) ,\mathrm{rev}\big(\gamma^{(n)}_1(y\!-\!T_n)\big) \right).
	\] 
	Since the distributions that determine block sizes in the transition kernels of Definition \ref{def:kernel:sp} are diffuse and by the 
	independence properties of $(\zeta(\ff^{(n)}),\gamma_1^{(n)},\gamma_2^{(n)})$, the longest interval in 
	$\beta^{T_{n+1}-}\!=\beta_1^{T_{n+1}-}\!\concat\beta_2^{T_{n+1}-}$ is a.s.\@ unique, and therefore
	\[
	\left(\mathrm{rev}( \beta_2^{T_{n+1}}), m^{T_{n+1}}, \mathrm{rev}(\beta_1^{T_{n+1}})\right)
	= \phi\left(\mathrm{rev}(\beta^{T_{n+1}-})\right)= \phi\left(\mathrm{rev}(\beta_2^{T_{n+1}-}) \concat \mathrm{rev}(\beta_1^{T_{n+1}-})\right).	
	\]	
	These observations show that $\big((\mathrm{rev}(\beta_2^y), m^y,  \mathrm{rev}(\beta_1^y)), y\ge 0\big)$ satisfies the definition of a $\cJ$-valued $\mathrm{SSIP}_{\!\dagger}^{(\alpha)}( \theta_2, \theta_1)$-evolution. 
\end{proof}

\subsection{The $\cI_H$-valued process.}
The $\cJ$-valued $\mathrm{SSIP}_{\!\dagger}^{(\alpha)}(\theta_1, \theta_2)$-evolution is only 
of secondary importance to us, as we are more interested in its associated $\cI_H$-valued process. 
However, there are two issues to be addressed.  
First, the definition of the $\cI_H$-valued process a priori depends on the initial choice of the ``middle'' block. Even with the natural choice of the longest block, the size of this block is typically exceeded by other blocks during the evolution. To establish the Markov property of the $\cI_H$-valued processes, we will view them as projections of $\cJ$-valued processes. Indeed, we will show that two $\cI_H$-valued ${\rm SSIP}^{(\alpha)}_\dagger(\theta_1,\theta_2)$-evolutions started from different choices of middle block indeed have the same law. Second, it is possible that the renaissance times $T_n$ accumulate, i.e.\@ $T_{\infty}<\infty$;   
we would like to understand the behaviour near the degeneration time $T_{\infty}$. 
The following two lemmas, whose proofs are postponed to Appendix \ref{sec:appx}, deal with these two problems.

\begin{lemma}~\label{lem:consist2}
	Let $\theta_1,\theta_2 \ge 0$. 
	Let  $((\beta_1^y, m^y,  \beta_2^y),\, y\ge 0)$ and $((\widetilde{\beta}_1^y,\widetilde{m}^y,  \widetilde{\beta}_2^y),\, y\ge 0)$ be two $\cJ$-valued $\mathrm{SSIP}_{\!\dagger}^{(\alpha)}(\theta_1, \theta_2)$-evolutions, with $(\beta^y, y\ge 0)$ and $(\widetilde{\beta}^y,\,y\ge 0)$ being their associated $\cI_H$-valued evolutions. 
	Suppose that $\beta^0= \widetilde{\beta}^0$, then we can couple these two processes such that $(\beta^y,\, y \ge 0) = (\widetilde{\beta}^y,\, y \ge 0)$ almost surely. 
\end{lemma}

\begin{lemma}\label{lem:dl}
		Let $\theta_1,\theta_2 \ge 0$. Let $(\beta^y,\, y\ge 0)$ be an $\cI_H$-valued  $\mathrm{SSIP}_{\!\dagger}^{(\alpha)}(\theta_1, \theta_2)$-evolution with renaissance times $T_n$, $n\ge 1$, and degeneration time $T_{\infty}$. 
	If $\mathbb{P}(T_\infty<\infty)>0$, then conditionally on $T_{\infty}<\infty$, 
	the total mass $\|\beta^{T_{n}}\|$ converges almost surely to zero as $n\rightarrow \infty$. 
	%we have $( \beta_1^y ,m^y,\beta_2^y)$ converges  a.s.\  to $(\emptyset, 0, \emptyset)$ for the metric $d_{\cJ}$, as $y\uparrow T_{\infty}$.
\end{lemma}

%\begin{remark}
%We will later see (Proposition~\ref{prop:dl}) that by appealing to Theorem~\ref{thm:mass-bis} below, the convergence in Lemma~\ref{lem:dl} in fact also holds without passing to a subsequence; however, we will require this weaker conclusion to study the total mass process and to prove Theorem~\ref{thm:mass-bis}. 
%\end{remark}

\begin{theorem}[Total mass of  an $\mathrm{SSIP}^{(\alpha)}_\dagger(\theta_1,\theta_2)$-evolution]\label{thm:mass-bis}
	For $\alpha\in (0,1)$ and $\theta_1,\theta_2 \ge 0$, 
	let $(\beta^y,\,y\ge 0)$ be an $\mathrm{SSIP}^{(\alpha)}_\dagger(\theta_1,\theta_2)$-evolution. 
		Then 
		$(\|\beta^y\|,\,y\ge 0)\sim{\tt BESQ}_{\|\beta^0\|}^\dagger(2\theta)$, where $\theta:=\theta_1+\theta_2-\alpha$.
		In other words, $(\|\beta^y\|,\,y\ge 0)$ is a ${\tt BESQ}_{\|\beta^0\|}(2\theta)$ killed at its first hitting time of zero.
\end{theorem}

To prove Theorem~\ref{thm:mass-bis}, we recall from \cite{PW18} a generalised additivity property of squared Bessel processes. 
\begin{lemma}[{\cite[Proposition~1]{PW18}}]\label{lem:besq}
	For any $\delta_1,\delta_2\in \bR$, and $b_1,b_2\ge 0$, let  $Z_1\sim \besq_{b_1}(\delta_1)$ and $Z_2\sim \besq_{b_2}(\delta_2)$ 
	be independent. Let $T$ be a stopping time relative to the filtration $(\cF_t,t \ge 0)$ generated by the pair of processes $(Z_1,Z_2)$, with $T\le \zeta(Z_1) \!\wedge\! \zeta(Z_2)$, where the lifetime $\zeta$ is defined as in \eqref{eq:besq-zeta}. Given $\cF_T$, let $Z_3\sim \besq_{Z_1(T)+Z_2(T)}(\delta_1\!+\!\delta_2)$. Then the process $Z$ defined as follows is a $\besq_{b_1+b_2}(\delta_1\!+\!\delta_2)$: 
	\[
	Z(t)=		\begin{cases}
	Z_1(t)+ Z_2(t),& \text{if} ~0\le t \le T,\\
	Z_3(t-T),&\text{if}~  t > T.
	\end{cases}         
	\]
\end{lemma}

\begin{proof}[Proof of Theorem~\ref{thm:mass-bis}]
	Let $W\sim \besq_{1}^\dagger(2 \theta)$ be independent of everything else. 
	Recall the construction described in Definition~\ref{defn:ipe}. 
	With notation therein, define for every $i\ge 1$ a process $Z_i$ by
	\begin{equation}\label{eq:Z}
	Z_i(x) =		\begin{cases}
	\|\beta^x\| ,& \text{if} ~0\le x \le T_i,\\
	\mathbf{1}\{\|\beta^{T_i}\|\!\ne\! 0\} \|\beta^{T_i}\| W\!\left((x\!-\!T_i)/\|\beta^{T_i}\| \right),&\text{if}~  x > T_i.
	\end{cases}         
	\end{equation}
	These processes are constructed on the same (large enough) probability space. 
	We first prove by induction that each $Z_i$ is a $\besq_{\|\beta^0\|}^\dagger(2 \theta)$. 
	Conditionally on $(\|\beta^x\|,\,x\le T_i)$, the process $\big(\mathbf{1}\{\|\beta^{T_i}\|\!\ne\! 0\}\|\beta^{T_i}\| W\!\left((x\!-\!T_i)/\|\beta^{T_i}\| \right),\,x\ge 0\big)$ has distribution $\besq^\dagger_{\|\beta^{T_i}\|}(2 \theta)$, by the scaling property of $\besq^\dagger(2 \theta)$. 
	Note that by Definition \ref{def:kernel:sp} and independence, $\|\beta^{T_i}\|=0$ only happens with positive probability when $\theta_1=\theta_2 =0$, and then there is $T_{j} = T_i$ and $Z_j \equiv Z_i$ for all $j\ge i$; in this case $2\theta=-2\alpha <0$ and $\besq_{0}^\dagger(-2 \alpha)$ and $\besq_0(-2\alpha)$ are the distribution of the constant zero process. 
	
	Let $\gamma_1^{(0)}$ and $\gamma_2^{(0)}$ be as in Definition~\ref{defn:ipe}. By Proposition~\ref{prop:Hunt+totalmass}, the total mass evolutions of $\gamma_1^{(0)}$ and $\gamma_2^{(0)}$ are independent $\besq_{ \|\beta_1^0\|}(2\theta_1)$  and $\besq_{ \|\beta_2^0\|}(2\theta_2)$ respectively, also independent of $\ff^{(0)}\sim \besq_{ \|m^0\|}(-2\alpha)$. 
	Noticing that 
	$\|\beta^x\|=\|\gamma_1^{(0)}(x)\|\!+ \!\mathbf{f}^{(0)}(x)\!+\! \|\gamma_2^{(0)}(x)\|$ for all $x\le T_1$, 
	we deduce from Lemma~\ref{lem:besq} that $Z_1 \sim \besq^\dagger_{\|\beta^0\|}(2 \theta)$, since we have $2\theta_1-2\alpha+2\theta_2= 2\theta$.   
	
	Suppose by induction that for some $i\ge 1$, for each $\mathrm{SSIP}_{\!\dagger}^{(\alpha)}(\theta_1,\theta_2)$-evolution  starting from any state in $\cJ$, its corresponding process $Z'_j$ as in \eqref{eq:Z} is a $\besq^\dagger(2 \theta)$ for each $j\le i$. 
	By the construction in Definition~\ref{defn:ipe}, conditionally on $(  \gamma_1^{(0)},\ff^{(0)}, \gamma_2^{(0)})$, the process 
	\[
	( \widetilde{\beta}_1^y, \widetilde{m}^y, \widetilde{\beta}_2^y):= (\beta_1^{T_1+y},m^{T_1+y}, \beta_2^{T_1+y}), \quad  y\ge 0,
	\]
	is an $\mathrm{SSIP}_{\!\dagger}^{(\alpha)}(\theta_1,\theta_2)$-evolution starting from $( \beta_1^{T_1},m^{T_1}, \beta_2^{T_1})$. 
	Define  
	\[
	\widetilde{Z}_i(x) =		\begin{cases}
	\left\|\beta^{T_i+x} \right\|,& \text{if} ~0\le x \le T_{i+1}-T_1,\\
	\mathbf{1}\left\{\|\beta^{T_{i+1}}\|\!\ne\! 0\right\} \left\|\beta^{T_{i+1}}\right\| W\left((x\!-\!T_{i+1})/\|\beta^{T_{i+1}}\| \right),&\text{if}~  x > T_{i+1}-T_1.
	\end{cases}         
	\]
	Then there is the identity 
	\[
	Z_{i+1}(x) =		\begin{cases}
	\|\beta^x\| ,& \text{if} ~0\le x \le T_i,\\
	\widetilde{Z}_i (x-T_i),&\text{if}~  x > T_i.
	\end{cases}         
	\]

	Given the triple $(\gamma_1^{(0)},  \ff^{(0)},\gamma_2^{(0)})$, by the inductive hypothesis $\widetilde{Z}_i$ has conditional distribution $\besq_{\|\beta^{T_1}\|}^\dagger(2 \theta)$. 
	%Therefore, the scaling property of squared Bessel processes yields that $(\|\beta^{T_1}\|^{-1} \widetilde{Z}_i( x \|\beta^{T_1}\| ), x\ge 0)$ has distribution $\besq_{1}(2 \theta)$, independent of the triple $( \gamma_1^{(0)}, \ff^{(0)},\gamma_2^{(0)})$. 
	Consequently, by the $i=1$ case, we have $Z_{i+1}\sim \besq_{\|\beta^0\|}^\dagger(2 \theta)$. This completes the induction step.  
	We conclude that $(Z_n)_{n\ge 1}$ is a sequence of processes, in which  
	each $Z_n$ is a $\besq_{\|\beta^0\|}^\dagger(2 \theta)$, and $(Z_n(y), y\le T_n) = (\|\beta^y\|,y\le T_n)$.

	We now prove the theorem for the case $\theta=\theta_1\!+\!\theta_2\!-\!\alpha < 1$. 
	%Recall from the proof of Lemma~\ref{lem:0} that we have built a sequence of processes $(Z_n)$ with 
	%each $Z_n\sim \besq_{\|\beta^0\|}(2 \theta)$, and $(Z_n(y), y\le T_n) = (\beta^y,y\le T_n)$ for the renaissance level $T_n$ of the process $(\beta^y)$. 
	For each $n\ge 1$, let $\tau_0(Z_n)$ be the  first hitting time of zero by $Z_n$. 
	Then $T_n\le \tau_0(Z_n)$ and $\tau_0(Z_n) -T_n =  \|\beta^{T_n}\| \tau_0(W)$, where $\tau_0(W)$ is the  first hitting time of zero by $W$ and has the law of $1/2 G$ with $G\sim \mathtt{Gamma}(1\!-\!\theta, 1)$. 
	We deduce that the distribution of the degeneration time $T_{\infty}$ is stochastically dominated by $\|\beta^0\|/2G$ and is thus a.s.\ finite. 
	%	indeed, we have by monotone convergence that, for any fixed $M>0$, 
	%	\[
	%	\bP (T_{\infty} >M) = \lim_{i\to \infty} \bP (T_{i} >M) \le  \bP (\zeta(Z_i) >M), 
	%	\]
	%	where $\zeta$ is the first hitting time to zero of a $\besq_{\|\beta^0\|} (2\theta)$. 
	%	Letting $M\to \infty$ yields that 
	%	 \[
	%	 \bP (T_{\infty} =\infty) =  \lim_{M\to \infty} \bP (T_{\infty} >M) \le \lim_{M\to \infty} \bP (\zeta >M)=0.
	%	 \] 
	Then it follows from Lemma~\ref{lem:dl} that a.s.\  $\lim_{n\to \infty} \|\beta^{T_{n}}\|=0$.  
	Therefore, $\tau_0(Z_{n})$ converges to $T_{\infty}$ a.s., as $n\to \infty$. 
	We conclude that $(Z_{n}(y\wedge\tau_0(Z_n)),\, y\ge 0)$ converges a.s.\ uniformly to $(\|\beta^{y\wedge T_\infty}\|,\, y\ge 0)$, and the limiting process is $\besq_{\|\beta^0\|}^\dagger(2\theta)$. 
	
	We finally study the case $\theta=\theta_1\!+\!\theta_2\!-\!\alpha \ge 1$. 
	By the connection between $\|\beta\|$ and the processes $(Z_n)_{n\ge 1}$, it suffices to prove that, for every $a>0$, $\bP(T_{\infty} <a) =0$. 
	On the event $\{T_{\infty} <a\}$, Lemma~\ref{lem:dl} leads to  $\lim_{n\to \infty} Z_{n} (T_{n}) = \lim_{n\to \infty} \|\beta^{T_{n}}\|=0$.  
	Therefore, for any  $\delta>0$, we have
	\[
	\{T_\infty <a\}
	\subseteq \bigcup_{n\in \bN} \{Z_{n} (T_{n}) <\delta , T_{\infty} <a\}. 
	\]
	Note that $\{Z_{n} (T_{n}) <\delta , T_{\infty} <a\}=\{\|\beta^{T_n} \| <\delta , T_{\infty} <a\}
	                                                                         \subseteq\{\inf_{y\in [0,T_n]}\|\beta^y\| <\delta, T_{\infty}<a\}=:A_n$.
	The sequence of events $(A_n, n\in \bN)$ is increasing.  
	By monotone convergence, we have
	\[
	\bP\left( \bigcup_{n\in\mathbb{N}} \big\{Z_{n} (T_{n}) \!<\!\delta , T_{\infty} \!<\!a \big\} \right)
	\le \lim_{n\to \infty} \bP(A_n) \le \limsup_{n\to \infty} \bP \left( \inf_{t\in [0,a]}\! Z_n(t) <\delta \right) = \bP \left( \inf_{t\in [0,a]}\! Z_1(t) <\delta \right), 
	\]
	using the fact that each $Z_{n}\sim \besq^\dagger_{\|\beta^0\|}(2\theta)=\besq_{\|\beta^0\|}(2 \theta)$. Since $2 \theta\ge 2$, 
	we have 
	\[
	\lim_{\delta\downarrow 0} \bP \Big( \inf_{t\in [0,a]} Z_1(t) <\delta \Big) =0. 
	\]
	As a result, we have $\bP(T_{\infty} <a) =0$. This completes the proof. 
\end{proof}

We have obtained the following dichotomy in the preceding proof. 
\begin{corollary}\label{prop:dl}
	For $\theta_1,\theta_2 \ge 0$, let $(\beta^y, y\ge 0)$ be an $\cI_H$-valued $\mathrm{SSIP}_{\!\dagger}^{(\alpha)}(\theta_1, \theta_2)$-evolution starting from $\beta^0\ne \emptyset$, with renaissance times $T_n$, $n\ge 0$, and degeneration time $T_{\infty}$. Set $\theta= \theta_1+\theta_2-\alpha$. 
	\begin{enumerate}
		\item If $\theta\ge 1$, then a.s.\ $T_{\infty}=\infty$ and $\beta^y\ne \emptyset$ for every $y\ge 0$.\vspace{0.2cm}  
		\item If $\theta<1$, then a.s.\ $T_{\infty}<\infty$ and $\lim_{y\uparrow T_{\infty}} \|\beta^y\| =0$. 
	\end{enumerate}	
\end{corollary}

\begin{proposition}\label{prop:JvaluedMP}
	For $\theta_1,\theta_2\!\ge\! 0$, a $\cJ$-valued $\mathrm{SSIP}_{\!\dagger}^{(\alpha)}(\theta_1, \theta_2)$-evolution is Borel right Markov on $(\cJ ,d_{\cJ})$, but not Hunt.
\end{proposition}

\begin{proof} 
	The $\mathcal{J}$-valued process takes values in $\mathcal{J}$, which is a Borel subset of a Polish space. 
	The c\`adl\`ag path property follows from the path-continuity of the $\mathrm{SSIP}^{(\alpha)}(\theta_1)$- and ${\rm RSSIP}^{(\alpha)}(\theta_2)$-evolutions used in the construction, and from Corollary~\ref{prop:dl}. 
	
	We know that squared Bessel processes, $\mathrm{SSIP}^{(\alpha)}(\theta_1)$- and $\mathrm{RSSIP}^{(\alpha)}(\theta_2)$-evolutions Borel right Markov processes. 	In the construction of a $\cJ$-valued process, we kill a Borel right Markov process with finite lifetime, and at the end of the lifetime, give birth to a new one according to a (deterministic) probability kernel.   
	This entails that a $\cJ$-valued process is itself a Borel right Markov process; see e.g.\ \cite[Th\'eor\`eme~II 3.18]{Bec07}.  
	
	Finally, the $\cJ$-valued process is not Hunt. 
	To see this, we consider the increasing sequence of stopping times $\tau_n:= \inf\{y\ge 0\colon m^y <1/n\}$, $n\ge 1$.  
	Then $\tau_n$ convergences to the first renaissance time $T_1$ as $n\to \infty$. 
	But the $\mathcal{J}$-valued process has a jump at $T_1$ with strictly positive probability; 
	so it is not quasi-left continuous.  
\end{proof}

Recall that Theorem \ref{thm:hunt} has three parts. The first claims that $\cI_H$-valued ${\rm SSIP}^{(\alpha)}_\dagger(\theta_1,\theta_2)$-evolutions are self-similar path-continuous Hunt processes, where self-similarity of an ${\rm SSIP}^{(\alpha)}_\dagger(\theta_1,\theta_2)$-evolution $(\beta^y,y\ge 0)$ means that 
for any $c>0$, the process $(c\odotip \beta^{c^{-1} y}, y\ge 0)$ is an $\mathrm{SSIP}_{\dagger}^{(\alpha)}(\theta_1, \theta_2)$-evolution starting from $c \odotip \beta^{0}$. The second part claims ${\tt BESQ}^\dagger(2\theta)$ total mass processes, where $\theta=\theta_1+\theta_2-\alpha$. The third part claims that the case $\theta_2=\alpha$ reduces to $\mathrm{SSIP}^{(\alpha)}(\theta_1)$-evolutions killed
at their first hitting time of $\emptyset$.

\begin{proof}[Proof of Theorem \ref{thm:hunt}]
	The space $(\mathcal{I}_H,d_H)$ is Polish by Lemma~\ref{lem:cI}. 
	The path-continuity for the $\mathcal{I}_H$-valued process between renaissance times follows from the path-continuity of any $\mathrm{SSIP}^{(\alpha)}(\theta_i)$-evolutions, $i=1,2$. The path-continuity at times $T_n$, $n\ge 1$, holds by construction since
	the parts of $\phi(\beta)$ in \eqref{eq:phi} have concatenation $\beta$. By Corollary~\ref{prop:dl}, the path-continuity at time $T_\infty$ 
	holds almost surely whenever $T_\infty<\infty$. 

	We have proved in Proposition~\ref{prop:JvaluedMP} that the corresponding $\cJ$-valued process is a Borel right Markov process. 
	By Lemma~\ref{lem:consist2}, Dynkin's criterion applies to the $\mathcal{I}_H$-valued process as a projection of the $\mathcal{J}$-valued process, so it is also a Borel right Markov process. The Hunt 
	property then follows from the path-continuity. 
	
		The same scaling property holds for $\besq(- 2\alpha)$, by \cite[A.3]{GoinYor03}, and for $\mathrm{SSIP}^{(\alpha)}(\theta_i)$ evolutions, $i=1,2$, by \cite[Theorem~1.4]{IPPAT} and hence also for $\mathrm{RSSIP}^{(\alpha)}(\theta_2)$. 
		Keeping track of Definition~\ref{defn:ipe}, we easily  
		check the self-similarity for a $\cJ$-valued $\mathrm{SSIP}_{\!\dagger}^{(\alpha)}(\theta_1, \theta_2)$-evolution.  Using~Lemma~\ref{lem:consist2}, we deduce the self-similarity of $\cI_H$-valued processes. 

        The claimed total mass process was established in Theorem \ref{thm:mass-bis}.

	For the final claim, we may assume that the $\mathrm{SSIP}_{\dagger}^{(\alpha)}(\theta_1,\alpha)$-evolution $(\beta^y, y\ge 0)$ is 
	associated with a $\cJ$-valued process $((\beta_1^y, m^y,  \beta_2^y), y\ge 0)$. 
	%For every $n\ge 1$, conditionally on $(\beta_1^{T_n}, m^{T_n}, \beta_2^{T_n})$, let $(\widetilde{\gamma}^{(n)}(y), y\ge 0)$ be an $\mathrm{SSIP}^{(\alpha)}(\theta_1)$-evolution starting from $\beta^{T_n}$. 
	%It follows from Lemma~\ref{lem:split} that the concatenated process
	%\[
	%\begin{cases}
	%\beta^y = \gamma_1^{(n-1)} (y) \concat (0, \ff^{(n-1)}(y)) \concat \gamma_2^{(n-1)} (y),	& \text{if}~ y\le  T_{n} -T_{n-1},\\
	%\widetilde{\gamma}^{(n)}(y), &\text{if}~ y>  T_{n} -T_{n-1},	\\
	%\end{cases}
	%\]
	%conditionally on $(\beta_1^{T_{n-1}}, m^{T_{n-1}}, \beta_2^{T_{n-1}})$, 
	%is an $\mathrm{SSIP}^{(\alpha)}(\theta_1)$-evolution starting from $\beta^{T_{n-1}}$. 
	In the notation of Definition~\ref{defn:ipe}, we have
	\[
	\beta^y = \gamma_1^{(n)} (y) \concat\big\{ \big(0, \ff^{(n)}(y)\big)\big\} \concat \gamma_2^{(n)} (y),	\qquad 0\le y\le  T_{n+1} -T_n,~ n\ge 0. \\
	\]	
	Conditionally on $(\beta_1^{T_n}, m^{T_n}, \beta_2^{T_n})$, the RHS	is an $\mathrm{SSIP}^{(\alpha)}(\theta_1)$-evolution starting from $\beta^{T_n}$, by Lemma \ref{lem:split} and Proposition \ref{prop:concat}. 
	Using this fact and Lemma~\ref{lem:dl}, the remainder of this proof is analogous to the proof of Theorem~\ref{thm:mass-bis}; details are left to the reader. 
\end{proof}

%====================================================
\section{Preliminaries on the clade construction of the two-parameter family ${\rm SSIP}^{(\alpha)}(\theta)$, $\alpha\in(0,1)$, $\theta\ge 0$.}\label{sec:prel:clades}

In Section \ref{sec:ip}, we postponed the proofs of Lemmas \ref{lem:consist2} and \ref{lem:dl} to Appendix \ref{sec:appx}. These proofs, as
well as the developments in Section \ref{sec:large} depend on clade constructions of \cite{Paper1-1,IPPAT} that we recall here.

\subsection{Construction of ${\rm SSIP}^{(\alpha)}(0)$ from marked stable L\'evy processes.}\label{sec:constr}

%We know from \cite{GoinYor03} that $\besq(- 2\alpha)$ has an exit boundary at zero. 
Let $\cE$ be the space of non-negative c\`adl\`ag excursions away from zero. 
For any $f\in \cE$, let $\zeta (f):= \sup \{ t\ge 0\colon f(t)>0\}$. Fix $\alpha\in (0,1)$.
Pitman and Yor \cite[Section~3]{PitmYor82} construct a sigma-finite excursion measure associated with $\besq(- 2\alpha)$ on (the subspace of continuous excursions in) the space $\cE$, which we scale to a measure $\nu^{(-2\alpha)}_{\mathtt{BESQ}}$ 
such that
\[
\nu^{(-2\alpha)}_{\mathtt{BESQ}} (f\colon \zeta(f) \ge y) = \frac{\alpha}{2^{\alpha} \Gamma(1-\alpha)\Gamma(1+\alpha)}  y^{-1-\alpha}, \qquad y>0,
\]
and under $\nu^{(-2\alpha)}_{\mathtt{BESQ}}$, conditional on $\zeta(f)=y$ for $0<y<\infty$, 
the process $f$ is the squared Bessel bridge with dimension parameter $4+2\alpha$, from $0$ to $0$ over $[0,y]$ (\cite[Section 11.3]{RevuzYor}).    
\cite[Section~3]{PitmYor82} offers several other equivalent descriptions of $\nu^{(-2\alpha)}_{\mathtt{BESQ}}$; see also \cite[Section~2.3]{Paper1-1}.

\begin{figure}
	\centering
	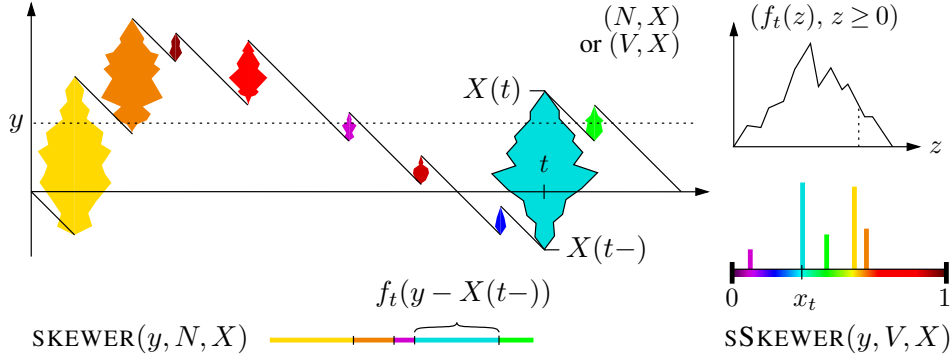
	\caption{A scaffolding $X$ with spindle marks: $(N,X)$, and the spindle $f_t$ at scaffolding time $t$. The skewer $\skewer(y,N,X)$ extracts 
	from each spindle that crosses level $y$ the spindle width at that level and builds an interval partition in which block sizes are spindle widths, 
	placed in left-to-right order without leaving gaps, as if on a skewer. The superskewer $\sskewer(y,V,X)$ is based on a richer point measure 
	$V$ that also yields further marks of allelic types (here a colour spectrum $[0,1]$) and builds a random measure in which atom sizes are spindle widths 
	and atom locations allelic types \cite{FVAT}.\label{fig:scaf-marks}}
\end{figure}

Following \cite{Paper1-1}, let $\fN$ be a Poisson random measure on $[0,\infty)\times \cE$ with intensity measure $\mathrm{Leb} \otimes \nu^{(-2\alpha)}_{\mathtt{BESQ}}$. 
Each atom of $\fN$, which is an excursion function in $\cE$, shall be referred to as a \emph{spindle}, in view of illustrations of $\fN$ as in Figure~\ref{fig:scaf-marks}. 
We pair $\fN$ with a \emph{scaffolding function} $\xi_{\fN}:=(\xi_{\fN}(t), t\ge 0)$ defined by
\begin{equation}\label{eq:scaffolding}
\xi_{\fN}(t):=  \lim_{z\downarrow 0} \left( 
\int_{[0,t]\times \{g\in \cE\colon \zeta(g)>z\}} \zeta(f)  \fN(ds,df) - \frac{(1+\alpha)t}{(2z)^{\alpha}\Gamma(1-\alpha)\Gamma(1+\alpha)}
\right).
\end{equation}
%such that the jump measure of $\xi(\fN)$ is given by $\fN (dt \times \zeta(f)\in dx)$. 
This is a spectrally positive stable L\'evy process of index $1+\alpha$, with L\'evy measure $\nu^{(-2\alpha)}_{\mathtt{BESQ}} (\zeta(f) \in  d y)$ and Laplace exponent $q^{1+\alpha}/2^{\alpha} \Gamma(1+\alpha)$, $q\ge 0$. 

For $x>0$, let $\ff\sim \besq_x(- 2\alpha)$, independent of $\fN$. Then $\ff$ is $\cE$-valued (actually taking values in the subspace of excursions that are continuous after an initial positive jump).   
Write $\mathtt{Clade}_x(\alpha)$ for the law of a \emph{clade of initial mass $x$}, which is a random point measure on $[0,\infty)\times \cE$ defined by
\begin{equation}\label{eq:clade}
\clade (\ff, \fN):= \delta(0,\ff)  + \fN\mid_{(0, T_{-\zeta(\ff)} (\fN)]\times \cE}, 
~\text{where}~ T_{-y} (\fN):= \inf\{t\ge 0 \colon \xi_{\fN}(t)= -y\}. 
\end{equation}
We also write $\mathrm{len}(\clade (\ff, \fN)):= T_{-\zeta(\ff)} (\fN)$ for its length, which is a.s.\@ finite. 
For $\gamma\in \cI_H$, let $(\fN_U, U\in \gamma)$ be a family of independent clades, with each $\fN_U \sim \mathtt{Clade}_{\mathrm{Leb}(U)} (\alpha)$. Then we define $\fN_{\gamma}$, a random point measure on $[0,\infty)\times \cE$, by the concatenation of $(\fN_U, U\in \gamma)$: 
\begin{equation}\label{eq:concatenation-clade}
\fN_{\gamma}:= \Concat_{U\in \gamma} \fN_U
:= \sum_{U\in\gamma} \int \delta (g(U)+t, f) \fN_{U} (dt,df), 
\quad \text{where}~ g(U):= \sum_{V\in\gamma\colon\! \sup V \le \inf U} \mathrm{len} (\fN_{V}). 
\end{equation}
We denote the distribution of $\mathbf{N}_\gamma$ by $\mathbf{P}_\gamma^{\alpha,0}$.

\begin{definition}[Skewer] \label{def:skewer}
	Let $N= \sum_{i\in \mathbb{N}} \delta(t_i,f_i)$ for some $(t_i,f_i)\in[0,T]\times\cE$ and $X\colon[0,T]\rightarrow\mathbb{R}$  
	c\`adl\`ag and such that 
	\[\sum_{t\in[0,T]\colon\!\Delta X(t)> 0} \delta(t, \Delta X(t)) = \sum_{i\in \mathbb{N}} \delta(t_i, \zeta(f_i)).\] 
	The \emph{skewer} of the pair $(N,X)$ at level $y$ is the interval partition
	\begin{equation}
	\skewer(y,N,X) :=     
	\{ (M^y(t-),M^y(t)) \colon M^y(t-)<M^y(t), t\in [0,T] \},
	\end{equation}
	where $M^y(t) = \int_{[0,t]\times\cE} f\big( y- X(s-) \big) N(ds,df)=\sum_{i\in\mathbb{N}\colon\! t_i\in[0,t]}f_i\big(y-X(t_i-)\big)$, $t\in[0,T]$, with the convention that $M^y(0-)=0$. 
	We denote the associated \em skewer process \em by $\skewerbar(N,X):= (\skewer(y,N,X),\, y\ge 0)$. 
	%	When $X = \xi(N)$, we write for simplicity $\skewerbar(N):= \skewerbar(N,\fX(N))$ and 
	%	$\skewer(y,N):= \skewer(y,N,\fX(N))$.
\end{definition}

See Figure \ref{fig:scaf-marks} for an illustration of $\skewer(y,N,X)$ in the natural extension of this definition to point measures with finitely many
atoms. 

\begin{proposition}[Theorem 1.8 of \cite{Paper1-2}]\label{prop:cladeconstr1}
	For $\gamma \in \cI_H$, let $\fN_{\gamma}\sim \mathbf{P}_\gamma^{\alpha,0}$ be as in \eqref{eq:concatenation-clade}. 
	Then $\skewerbar(\fN_{\gamma}, \xi_{\fN_{\gamma}})$ is an $\mathrm{SSIP}^{(\alpha)}(0)$-evolution starting from $\gamma$. 
\end{proposition}

Note that the skewer process is indexed by the non-negative levels of the scaffolding function, while the time axis of the scaffolding function induces the left-to-right order within the interval partition. This means we need to be careful when we refer to ``time'', so it will often be clearest if we distinguish ``scaffolding time $t$'' and ``level $y$'', but we will also continue to refer to ``time $y$'' in an ${\rm SSIP}^{(\alpha)}(0)$, which in the context of a clade construction is the same as ``level $y$''.

An $\mathrm{SSIP}^{(\alpha)}(0)$-evolution is also called a \emph{type-1 evolution} in \cite{Paper1-2}, in view of its role in the
construction of the Aldous diffusion \cite{Paper4}, where type-$j$ evolutions have $j$ top masses (leftmost blocks), $j=0,1,2$.   

The Markov property of $\mathrm{SSIP}^{(\alpha)}(0)$-evolutions at time $y>0$ corresponds to a decomposition of $\mathbf{N}_\gamma$
at scaffolding level $y$ in $\xi_{\mathbf{N}_\gamma}$. Indeed, the $\mathrm{SSIP}^{(\alpha)}(0)$-evolution after time $y$ is described
by the the spindles of $\mathbf{N}_\gamma$ that the $\xi_{\mathbf{N}_\gamma}$ places above level $y$, and these spindles naturally split into clades according to
the excursions of $\xi_{\mathbf{N}_\gamma}$ above level $y$, which (for all excursions above almost all levels) start by a jump that corresponds 
to the upper part of a spindle straddling level $y$. 

More precisely, denote by $S_\gamma(t)={\rm Leb}(\{u\le t\colon\xi_{\mathbf{N}_\gamma}(u)\ge y\})$ the amount of time up to scaffolding time 
$t\in[0,{\rm len}(\mathbf{N}_\gamma)]$ that the scaffolding $\xi_{\mathbf{N}_\gamma}$ has spent above level $y$. Then we define the associated point measure
\[N^{\ge y}(\mathbf{N}_\gamma)
    :=\sum_{\text{points } (t,f_t) \text{ of }\mathbf{N}_\gamma}
                        \left(\mathbf{1}\{\xi_{\mathbf{N}_\gamma}(t-)\ge y\}\delta(S_\gamma(t),f_t)
                            +\mathbf{1}\{\xi_{\mathbf{N}_\gamma}(t-)<y<\xi_{\mathbf{N}_\gamma}(t)\}\delta(S_\gamma(t),\hat{f}^y_t)\right),
\]
where $\hat{f}_t^y(s)=f(y-\xi_{\mathbf{N}_\gamma}(t-)+s)$, $s\in[\xi_{\mathbf{N}_\gamma}(t)-y]$, is the part of the spindle $f_t$ above level $y$. We can similarly define $N^{\le y}(\mathbf{N}_\gamma)$ based on (part-)spindles below level $y$ and refer to the $\sigma$-algebra generated by $N^{\le y}(\mathbf{N}_\gamma)$
as the \em history below level $y$\em. Then the point measures $\mathbf{N}_\gamma$ satisfy the following property, which we refer to as 
\em Markov-like property.\em  

\begin{proposition}[Proposition 6.6 of \cite{Paper1-1}]\label{prop:markovlike1} Let $\mathbf{N}_\gamma\sim\mathbf{P}^{\alpha,0}_\gamma$ and $y>0$. Then conditionally given the history below level
$y$, we have $N^{\ge y}(\mathbf{N}_\gamma)\sim\mathbf{P}^{\alpha,0}_{\beta^y}$, where $\beta^y:=\skewer(y,\mathbf{N}_\gamma,\xi_{\mathbf{N}_\gamma})$.
\end{proposition}

\subsection{Clade construction of $\mathrm{SSIP}^{(\alpha)}(\theta)$-evolutions with $\theta>0$.}
\label{sec:alphatheta}
Let $\alpha\in (0,1)$ and $\theta>0$. The clade construction of $\mathrm{SSIP}^{(\alpha)}(\theta)$-evolutions is a Poissonian construction,
based on a different type of clade, clades without an initial spindle. To define these clades, first consider again a Poisson random measure $\fN$ on 
$[0,\infty)\times \cE$ with intensity measure $\mathrm{Leb} \otimes \nu^{(-2\alpha)}_{\mathtt{BESQ}}$ and $\xi_{\fN}$ its associated 
scaffolding defined as in \eqref{eq:scaffolding}. 
For $y>0$, let $T^{-y} := \inf\{t\ge 0\colon \xi_{\fN} (t)<-y\}$, and $T^{(-y)-}:= \sup_{z\in(0,y)}T^{-z}$. 
We recall from \cite[Section~2.3]{IPPAT} that there is a sigma-finite measure $\nu^{(\alpha)}_{\mathrm{\bot cld}}$ on a suitable space $(\cN,\Sigma(\cN))$ of counting measures on $[0,\infty)\times \cE$ that can be defined as
\[
\nu^{(\alpha)}_{\mathrm{\bot cld}} (A):= \bE \left[ \sum_{y\in [0,1]} \mathbf{1} \left\{ \fN|^{\leftarrow}_{[T^{(-y)-}, T^{-y})} \in A  \right\} \right], \qquad A\in\Sigma(\cN), 
\]
where $\fN|^{\leftarrow}_{[a, b)}:=  \sum_{(s,f) ~\text{points of}~ \fN\colon\! s\in [a,b)} \delta(s-a, f)$ for every $a<b$. Since each interval
$[T^{(-y)-},T^{-y})$ is the interval of an excursion of $\xi_{\fN}$ above the minimum, the sigma-finite measure 
$\nu^{(\alpha)}_{\mathrm{\bot cld}}$ captures the distribution of the associated point measure of spindles during such an excursion under the It\^o excursion measure of the spectrally positive stable L\'evy process $\xi_{\fN}$ reflected at its running minimum process.
See \cite[Section~2.3]{IPPAT} for more detailed discussion on $\nu^{(\alpha)}_{\mathrm{\bot cld}}$. 

\begin{lemma}[Equations (2.17)--(2.18) in \cite{IPPAT}]\label{lem:descent} Let $\mathbf{N}_x=\delta(0,\mathbf{f})+\mathbf{N}_x^\circ\sim{\tt Clade}_x(\alpha)$ be 
  a clade of initial mass $x$ and write $T^{-y}:=\inf\{t\ge 0\colon\xi_{\mathbf{N}^\circ_x}(t)<-y\}$, $y\in[0,\zeta(\mathbf{f}))$. 
  Then conditionally given $\zeta(\mathbf{f})=s$,
  \[
    \cev{\mathbf{F}}_x:=\sum_{y\in[0,\zeta(\mathbf{f}))\colon\!T^{(-y)-}<T^{-y}}\delta\left(\zeta(\mathbf{f})-y,\mathbf{N}_x^\circ|^{\leftarrow}_{[T^{(-y)-},T^{-y})}\right),
  \]
  is a Poisson random measure on $[0,s]\times\mathcal{N}$ with intensity measure ${\rm Leb}\otimes\nu_{\perp{\rm cld}}^{(\alpha)}$. 
\end{lemma}

Let $\cev{\fF}$ be a Poisson random measure on $[0,\infty)\times \mathcal{N}$ with intensity measure $(\theta/\alpha)\mathrm{Leb}\otimes  \nu^{(\alpha)}_{\mathrm{\bot cld}}$. We will abbreviate the distribution of $\cev{\fF}$ to $\mathbf{P}^{\alpha,\theta}_\emptyset$. We will
now use the points $(s,N_s)$ of $\cev{\fF}$ to provide immigration at level $s$ that contributes to levels $y\ge s$ via 
$\skewer(y-s,N_s,\xi_{N_s})$.

\begin{proposition}[Proposition 3.4 of \cite{IPPAT}]\label{prop:cladeconstr2} Let $\cev{\fF}\sim\mathbf{P}^{\alpha,\theta}_{\emptyset}$ and 
\begin{equation}\label{eq:cev-construction}
\cev{\beta}^y = \Concat_{\textrm{\rm points }(s, N_s)\textrm{ \rm of }\cev{\fF}\colon\! s\in [0, y]\downarrow} \skewer(y-s, N_s,\xi_{N_s} ), \qquad y\ge 0. 
\end{equation}
By $\downarrow$ we stress that this concatenation is in decreasing order of $s$, setting skewers of $N_s$ with $s$ higher to the left of those with $s$ lower. Then $(\cev{\beta}^y,y\ge 0)$ is an ${\rm SSIP}^{(\alpha)}(\theta)$-evolution starting from $\emptyset$. 
\end{proposition}

Recall from Proposition \ref{prop:concat} that an ${\rm SSIP}^{(\alpha)}(\theta)$-evolution starting from $\beta^0\in\cI_H$ can be constructed
as $(\beta_1^y\concat\beta_2^y,\,y\ge 0)$ for an independent pair of an ${\rm SSIP}^{(\alpha)}(\theta)$-evolution $(\beta_1^y,\,y\ge 0)$
starting from $\emptyset$ and an ${\rm SSIP}^{(\alpha)}(0)$-evolution $(\beta_2^y,\,y\ge 0)$ starting from $\beta^0$. In particular, we can
combine Propositions \ref{prop:cladeconstr1} and \ref{prop:cladeconstr2} to obtain a clade construction of an ${\rm SSIP}^{(\alpha)}(\theta)$-evolution starting from $\beta^0$, based on an independent pair $(\cev{\fF},\fN_{\beta^0})\sim\mathbf{P}^{\alpha,\theta}_\emptyset\otimes\mathbf{P}^{\alpha,0}_{\beta^0}$.

Finally, let us note the Markov-like property in this context. To this end, we denote by 
\[N^{\ge y}(\cev{\mathbf{F}})=\Concat_{\text{points }(s, N_s)\text{ of }\cev{\fF}\colon\! s\in [0, y]\downarrow}N_s^{\ge y-s}\]
the point measure of spindles that takes into account all immigration at levels $s\le y$ and builds a point measure from the associated (part-)spindles above level $y=s+(y-s)$. We collect the immigration at levels $s>y$ in a point measure $\cev{\mathbf{F}}^{\ge y}:=\cev{\mathbf{F}}|^{\leftarrow}_{(y,\infty)}$ on $[0,\infty)\times\mathcal{N}$. We include the remaining spindles below level $y$, captured
in $N_s^{\le y-s}$, from immigration at levels $s\le y$, in the history below level $y$.

\begin{proposition}[Lemma 3.10 of \cite{IPPAT}]\label{prop:markovlike2} 
  Let $(\cev{\fF},\fN_{\beta^0})\sim\mathbf{P}^{\alpha,\theta}_\emptyset\otimes\mathbf{P}^{\alpha,0}_{\beta^0}$ and $y>0$. Denote by
  $(\beta^y,\,y\ge 0)$ the associated ${\rm SSIP}^{(\alpha)}(\theta)$-evolution. Then conditionally given the history below level $y$, we have 
  $(\cev{\mathbf{F}}^{\ge y},N^{\ge y}(\cev{\mathbf{F}})\concat N^{\ge y}(\mathbf{N}_{\beta^0}))\sim \mathbf{P}^{\alpha,\theta}_\emptyset\otimes\mathbf{P}^{\alpha,0}_{\beta^y}$.
\end{proposition}

For more details, we refer to \cite[Section 3]{IPPAT}.

\subsection{Distributional properties of ${\rm SSIP}^{(\alpha)}(\theta)$-evolutions.}

Fix $\theta>0$. Consider an ${\rm SSIP}^{(\alpha)}(\theta)$-evolution $(\cev\beta^y,\,y\ge 0)$ starting from $\emptyset$, constructed as in Proposition \ref{prop:cladeconstr2}. We know from \cite[Proposition~3.6]{IPPAT} that the marginal distributions exhibit pseudo-stationarity for all $y>0$ 
\begin{equation}\label{eq:prm-marginal}
\cev\beta^y\sim \mathtt{Gamma}(\theta, (2y)^{-1})\cdot \pdip^{(\alpha)}(\theta). 
\end{equation}
For future usage, we explore in more detail the components of $\cev\beta^y$. 
By \cite[proof of Proposition~3.6]{IPPAT}, the clades of $\cev{\mathbf{F}}$ that survive to level $y$, for fixed $y>0$, can be listed as $N^y_k$ 
with immigration level $s^y_k$, $k\ge 1$, ordered in increasing order of immigration levels, which accumulate at $y$. This corresponds to the order from right to left in the concatenation of \eqref{eq:cev-construction}. Moreover, the resulting decomposition of $\cev\beta^y$ in \eqref{eq:cev-construction} can be expressed in terms of families of independent identically distributed (i.i.d.) random variables, as follows: 
\begin{equation}\label{eq:ssip-clades}
\Big(\skewer(y-s^y_k, N^y_k, \xi_{N^y_k}),\ \  k\ge 1\Big)
\ed \Big(E_k \prod_{i=1}^k B_i\bar\beta_k,\ \   k\ge 1\Big)
\ed \Big(G(1-B_k) \prod_{i=1}^{k-1} B_i\bar\beta_k,\ \   k\ge 1\Big),  
\end{equation}
where $(\bar\beta_i)_{i\ge 1}$ are i.i.d.\@ $\mathtt{PDIP}^{(\alpha)}( 0)$, $(B_i)_{i\ge 1}$  are i.i.d.\@ $\mathtt{Beta}(\theta- \alpha, 1)$, $(E_i)_{i\ge 1}$ are i.i.d.\@ $\mathtt{Exp}((2y)^{-1})$ and $G\sim\mathtt{Gamma}(\theta, (2y)^{-1}) $, independent of each other.
Fix any $j\ge 1$.  
Let $G'\sim\mathtt{Gamma}(\theta, (2y)^{-1})$ and $\bar{\gamma}\sim \mathtt{PDIP}^{(\alpha)}(\theta)$ be independent, further independent of $(E_k,B_k, \bar\beta_k)_{k\le j}$. 
It follows that 
\begin{align}
&\left( \Concat_{i=\infty}^{j+1} \skewer(y-s^y_i, N^y_i, \xi_{N^y_i}),\quad  \Big(\skewer(y-s^y_k, N^y_k, \xi_{N^y_k})\Big)_{k\le j}\right)\notag\\
&\ed \left( \prod_{i=1}^j B_{i} \Concat_{k=\infty}^{1} E_{j+k} \prod_{i=1}^k B_{j+i}\bar\beta_{j+k}, \quad \left(E_k \prod_{i=1}^k B_i\bar\beta_k \right)_{k\le j}\right)
\ed   \left( G' \prod_{i=1}^j B_{i} \bar{\gamma}, \quad \left(E_k\prod_{i=1}^k B_i\bar\beta_k\right)_{k\le j}\right),\label{eq:ssip-clades-bis}
\end{align}
where the second equality is from the observation that $\Concat_{k=\infty}^{1} E_{j+k} \prod_{i=1}^k B_{j+i}\bar\beta_{j+k}$ is independent of $(E_k,B_k, \beta_k)_{k\le j}$ and has distribution $\mathtt{Gamma}(\theta, (2y)^{-1})\cdot \pdip^{(\alpha)}(\theta)$ by \eqref{eq:ssip-clades} and \eqref{eq:prm-marginal}. 

It will be useful to also describe the distribution of \eqref{eq:ssip-clades} conditionally given $(s^y_k,\,k\ge 1)$. Specifically, we know from 
\eqref{eq:pdip:0-alpha} and \cite[Lemma~3.5]{IPPAT}, that conditionally on $(s^y_k,\,k\ge 1)$, the interval partitions 
$\skewer (y -s^y_k, N^y_k,\xi_{N^y_k})$, $k\ge 1$, are conditionally independent, with 
\begin{equation}\label{eq:conddist}\skewer (y -s^y_k, N^y_k,\xi_{N^y_k})\ed\{(0,H_k^y)\} \concat \gamma_k^y,\qquad k\ge 1,
\end{equation}
where $H^y_k\sim \mathtt{Gamma}(1-\alpha, (2 (y-s^y_k))^{-1})$ and 
$\gamma_k^y\sim \mathtt{Gamma}(\alpha, (2 y-s^y_k)^{-1}) \cdot\pdip^{(\alpha)}(\alpha)$ are independent, $k\ge 1$. In particular, we note that the
dependence on $\theta>0$ is only via the increasing sequence of immigration levels $s^y_k$, $k\ge 1$, of clades surviving to level $y$, whose distribution is given in \cite[proof of Proposition~5.1]{FVAT} as
\begin{equation}\label{eq:immlev} \sum_{k\ge 1}\delta(s^y_k)\quad\mbox{is a Poisson random measure on $[0,y)$ with intensity 
$\theta(y-s)^{-1}ds$.}
\end{equation}

\section{Construction of $\mathrm{SSIP}(\theta_1,\theta_2)$-evolutions for $\theta_1\ge\alpha$, $\theta_2\ge 0$, and the proof of Theorem \ref{thm:pseudo}.}\label{sec:large}

In this section, we make precise the construction of $\mathrm{SSIP}^{(\alpha)}(\theta_1, \theta_2)$-evolutions in the case 
$\theta_1\ge \alpha$. The approach does not depend on the construction of ${\rm SSIP}^{(\alpha)}_\dagger(\theta_1,\theta_2)$-evolutions
nor indeed on any developments in Section \ref{sec:ip}, except when we establish the connections between the two processes in Section \ref{sec:stop}.

\subsection{Definition, Markov property and path-continuity of $\mathrm{SSIP}^{(\alpha)}(\theta_1, \theta_2)$-evolutions with $\theta_1\ge \alpha$.}

Fix $\theta_1\ge\alpha$ and let $\underline{\cev{\fF}}\sim\mathbf{P}^{\alpha,\theta_1-\alpha}_\emptyset$ be a Poisson random measure on $[0,\infty)\times \mathcal{N}$ with intensity $(\theta_1/\alpha\,-1)\mathrm{Leb}\otimes  \nu^{(\alpha)}_{\mathrm{\bot cld}}$. 
Rather than using \eqref{eq:cev-construction}, we modify this construction and define $(\underline{\cev{\beta}}^y, y\ge 0)$ by left-right-reversing each immigrating clade:
%\begin{equation}\label{eq:theta-0}
%\cev{\beta}_{\theta}^y:= \Concat_{t\in (-y\theta/\alpha,0)} \int_{\mathcal{N}}  \mathrm{rev}\left(\skewer\big(y - \frac{\alpha}{\theta} |t|, N\big)\right) d \cev{\mathbf{F}} (t, N), \qquad y \ge 0.
%\end{equation}
\begin{equation}\label{eq:theta-0}
\underline{\cev{\beta}}^y:= \Concat_{\text{points}~ (s,N_s) ~\text{of}~\underline{\cev{\fF}}\colon\! s \in [0,y]\downarrow}  \mathrm{rev}\left(\skewer\big( y -s, N_{s} ,\xi_{N_s}\big)\right) , \qquad y \ge 0.
\end{equation}

\begin{definition}[$\mathrm{SSIP}^{(\alpha)}(\theta_1, \theta_2)$-evolution with $\theta_1\ge \alpha$]~\label{defn:theta-ge}
	For $\theta_1 \ge \alpha$, $\theta_2\ge 0$, and $\gamma \in \cI_H$. Let $(\vecc{\beta}^y ,y\ge 0)$ be an  $\mathrm{RSSIP}^{(\alpha)}(\theta_2)$-evolution starting from $\gamma$, and define $(\underline{\cev{\beta}}^y,y\ge 0)$ as in \eqref{eq:theta-0}, based on  $\underline{\cev{\fF}}\sim\mathbf{P}^{\alpha,\theta_1-\alpha}_\emptyset$.
	Then the $\cI_H$-valued process 
	\[
	\beta^y:= \underline{\cev{\beta}}^y \concat \vecc{\beta}^y, \qquad y\ge 0, 
	\]
	is an \emph{$\mathrm{SSIP}^{(\alpha)}(\theta_1, \theta_2)$-evolution} starting from $\gamma$. 
\end{definition}

\begin{proposition}[Total mass]
	The total mass of an $\mathrm{SSIP}^{(\alpha)}(\theta_1, \theta_2)$-evolution is a $\besq (2\theta)$ with $\theta:= \theta_1 +\theta_2 -\alpha$. 
\end{proposition}
\begin{proof}
	$(\underline{\cev{\beta}}^y,\,y\ge 0)$ has the same total mass as an $\mathrm{SSIP}^{(\alpha)} (\theta_1-\alpha)$-evolution, which is a $\besq(2(\theta_1-\alpha))$ starting from zero. Then we conclude by the additivity of squared Bessel processes. 
\end{proof}

\begin{proposition}[Self-similarity]
	Let  $(\beta^y,\,y\ge 0)$ be an $\mathrm{SSIP}^{(\alpha)}(\theta_1, \theta_2)$-evolution starting from $\gamma$. Then for every $c>0$, the process 
	$(c\odotip\beta^{y/c},y\ge 0)$ is an $\mathrm{SSIP}^{(\alpha)}(\theta_1, \theta_2)$-evolution starting from $c\odotip \gamma$. 
\end{proposition}
\begin{proof}
	Since an $\mathrm{RSSIP}^{(\alpha)}(\theta_2)$-evolution possesses this scaling property \cite[Theorem~1.4(ii)]{IPPAT}, it suffices to prove 
	that 
	\begin{equation}\label{eq:cev-scaling}
	(\underline{\cev{\beta}}^y ,y \ge 0)\ed (c \odotip \underline{\cev{\beta}}^{y/c}, y\ge 0)
	\end{equation} 
	with $(\underline{\cev{\beta}}^y,\,y\ge 0)$ defined as in \eqref{eq:theta-0}. 
	
	The proof of \eqref{eq:cev-scaling} is similar to that of \cite[Theorem~1.4(ii)]{IPPAT}.
	For any point $(s,N_s)$ of $\underline{\cev{\fF}}$, we define a point measure $N'_{cs}$ 
	by replacing each atom $(t, f)$ of $N_{s}$ with $(c^{1+\alpha} t, c f (\cdot/c))$. 
	We know from \cite[Lemma~2.11]{IPPAT} that 
	$ \underline{\cev{\fF}}':=\sum_{\text{points}~ (s,N_s) ~\text{of}~\underline{\cev{\fF}}} \delta(cs, N_{cs}')\ed \underline{\cev{\fF}}$. So 
	$(\underline{\cev{\beta}}^y,\,y\ge 0)$  has the same law as the process
\[
\underline{\cev{\beta}}'^{y} := \Concat_{\text{points}~ (r,N'_r) ~\text{of}~\underline{\cev{\fF}}'\colon\!r \in [0,y]\downarrow}  \mathrm{rev}\left(\skewer\big( y - r, N'_{r} ,\xi_{N'_r}\big)\right) ,\qquad y\ge 0. 
\]
On the other hand, because of the identity
	\[
	\skewer\big(c y -cs, N'_{cs} ,\xi_{N'_{cs}}\big) 
	=  c \cdot \skewer\big( y -s, N_{s} ,\xi_{N_{s}}\big) , \quad s \in [0,y],\ y>0,
	\]
	we have
$
\underline{\cev{\beta}}'^{y} 
= c\odotip\underline{\cev{\beta}}^{y/c}$, $y \ge 0$. 
	As a result, we have \eqref{eq:cev-scaling}, completing the proof. 
\end{proof}

\begin{proposition}[Markov property]\label{prop:Markov-ge}
	Let $(\beta^y,y\ge 0)$ be an $\mathrm{SSIP}^{(\alpha)}(\theta_1, \theta_2)$-evolution starting from $\gamma$. 
	For any $y\ge 0$, conditionally on $(\beta^x, x\le y)$ the process 
	$(\beta^{z+y}, z\ge 0)$ is an $\mathrm{SSIP}^{(\alpha)}(\theta_1, \theta_2)$-evolution
	starting from $\beta^y$. 
\end{proposition}

\begin{proof}
	Because of the Markov property \cite[Theorem~1.4]{IPPAT} of an $\mathrm{SSIP}^{(\alpha)}(\theta_2)$-evolution, 
	it suffices to prove the statement for the case $\theta_2 =0$ and $\gamma= \emptyset$.
	
	Fix $y\ge 0$, we have the decomposition 
	$\beta^{y+z} = \underline{\cev{\beta}}^{y,z} \concat \underline{\vecc{\beta}}^{y,z}, $
	where 
	\[
	\underline{\cev{\beta}}^{y, z}:= \Concat_{\text{points}~ (s,N_s) ~\text{of}~\underline{\cev{\fF}}\colon\! s \in (y,y+z]\downarrow}   \mathrm{rev}\left(\skewer\big(y+z -  s, N_{s },\xi_{N_{s}} \big)\right) , \qquad z \ge 0,
	\]
	and
	\[
	\underline{\vecc{\beta}}^{y,z}:= \Concat_{\text{points}~ (s,N_s) ~\text{of}~\underline{\cev{\fF}}\colon\! s \in [0,y]\downarrow}   \mathrm{rev}\left(\skewer\big(y+z - t,N_s,\xi_{N_s}\big)\right)  , \qquad z \ge 0.
	\]
	% Almost surely there are a finite number of atoms $(t_i, N_i)_{i\le k}$ of $\cev{\fF}$ such that 
	%$\zeta(N_i) \ge y -\frac{\alpha}{\theta}|t_i|$. 
	It follows from the Markov-like property of $\underline{\cev{\fF}}$ recalled in Proposition \ref{prop:markovlike2} above that the latter process is an  $\mathrm{RSSIP}^{(\alpha)} (0)$-evolution starting from $\beta^y$, 
	whereas the Poisson property shows that the former process has 
	the same law as $(\underline{\cev{\beta}}^z,\,z\ge 0)$ and is independent of $(\underline{\vecc{\beta}}^{y,z},\,z\ge 0)$. This leads to the desired statement. 
\end{proof}

\begin{proposition}[Path-continuity]
An $\mathrm{SSIP}^{(\alpha)}(\theta_1, \theta_2)$-evolution a.s.\ has continuous paths. 
\end{proposition}
\begin{proof}
It follows from the same arguments as in \cite[proof of Proposition~3.2]{IPPAT} that $(\underline{\cev{\beta}}^y,\,y\ge 0)$ defined as in \eqref{eq:theta-0} a.s.\ has continuous paths in $(\cI_H, d_{H})$. 
Combining this fact with the path-continuity of an $\mathrm{RSSIP}^{(\alpha)}(\theta_2)$-evolution, we deduce the claim. 
\end{proof}

\subsection{Identification of the two-parameter family of $\mathrm{SSIP}^{(\alpha)}(\theta)$-evolutions, $\alpha\in(0,1)$, $\theta\ge 0$.}

%--------------------------
\begin{proposition}\label{prop:1-ge}
	An $\mathrm{SSIP}^{(\alpha)}(\theta_1,\alpha)$-evolution starting from $\gamma$ 
	is an $\mathrm{SSIP}^{(\alpha)}(\theta_1)$-evolution starting from $\gamma$. 
\end{proposition}
In preparation of proving this statement, we consider a Poisson random measure $\cev{\fF}^{(\mathrm{r})}$ on $[0,\infty)\times \mathcal{N}$ with intensity $(\theta_1/\alpha-1)\mathrm{Leb}\otimes  \nu^{(\alpha)}_{\mathrm{\bot cld}}$, whose clades we label as \emph{red}, and an independent Poisson random measure $\cev{\fF}^{(\mathrm{b})}$ on $[0,\infty)\times \mathcal{N}$ with intensity $\mathrm{Leb}\otimes  \nu^{(\alpha)}_{\mathrm{\bot cld}}$, whose clades we label as \emph{blue}.  
Let 
\[
\cev\beta^z = \Concat_{(s,N_s) ~\text{points of}~\cev{\fF}^{(\mathrm{r})}+\cev{\fF}^{(\mathrm{b})}\colon\! s\in [0, z]\downarrow} \skewer(z-s, N_s,\xi_{N_s} ), \qquad z\ge 0. 
\]
Then $(\cev\beta^z, z\ge 0)$ is an $\mathrm{SSIP}^{(\alpha)} (\theta_1)$-evolution, by Proposition \ref{prop:cladeconstr2}. We seek to compare this process  with an $\mathrm{SSIP}^{(\alpha)} (\theta_1, \alpha)$-evolution; to this end, let us explore this two-colour model in more detail. 

Fix any $y\ge 0$. The distribution of the immigration levels $(s_k^y,\,k\ge 1)$ of \emph{red} clades in $\cev{\fF}^{(\mathrm{r})}$ surviving to level $y$ is given in \eqref{eq:immlev}, with $\theta=\theta_1-\alpha$. Writing $N_k^y$ for the red clade immigrating at level $s_k^y$, we also read from \eqref{eq:conddist} the conditional distribution given $(s_k^y,\,k\ge 1)$ of  
%Then $(\Delta^y_i= y - t_i^y)_{i\ge 1}$ form atoms of a Poisson random measure with intensity
%\begin{equation}\label{eq:Delta_i}
%(\theta_1- \alpha) x^{-1} d x,  \qquad  x \in (0,  y]. 
%\end{equation}
the interval partition $\skewer (y -s^y_k, N^y_k)\ed\{(0,H_k^y)\} \concat \gamma_k^y$. 

Given the immigration levels $(s^y_k,\,k\ge 1)$ of the red clades and $s^y_0:= 0$, let $\mu^y_k$ be the contribution at level $y$ of the \emph{blue} clades that are immigrating at levels in the interval $[s^y_{k-1}, s^y_{k})$,  i.e.\
\[
\mu^y_k:= \Concat_{(s, N_s) ~\text{points of}~\cev{\fF}^{(\mathrm{b})}\colon\! s\in [s^y_{k-1}, s^y_{k})\downarrow} \skewer(y-s, N_s,\xi_{N_s} ), \quad k\ge 1. 
\] 
Note that $\mu^y_k$ may be empty. Then we have a decomposition of $\cev\beta^y$, as illustrated in Figure~\ref{fig:MarkClades}:  
\begin{equation}\label{eq:gammamudecomp}
\cev\beta^y = \Concat_{k=\infty}^1 \Big( \{(0,H_k^y)\} \concat \gamma_k^y \concat \mu_k^y \Big).
\end{equation}
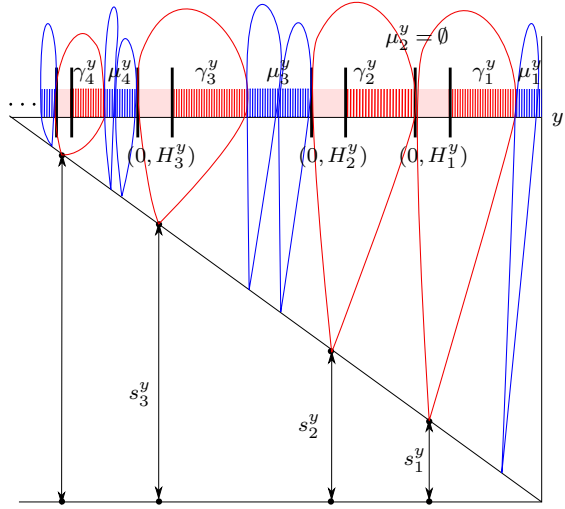
\begin{figure}[t]
	\centering
\begin{picture}(300,190)
		 \put(0,0){\includegraphics[width=0.45\linewidth]{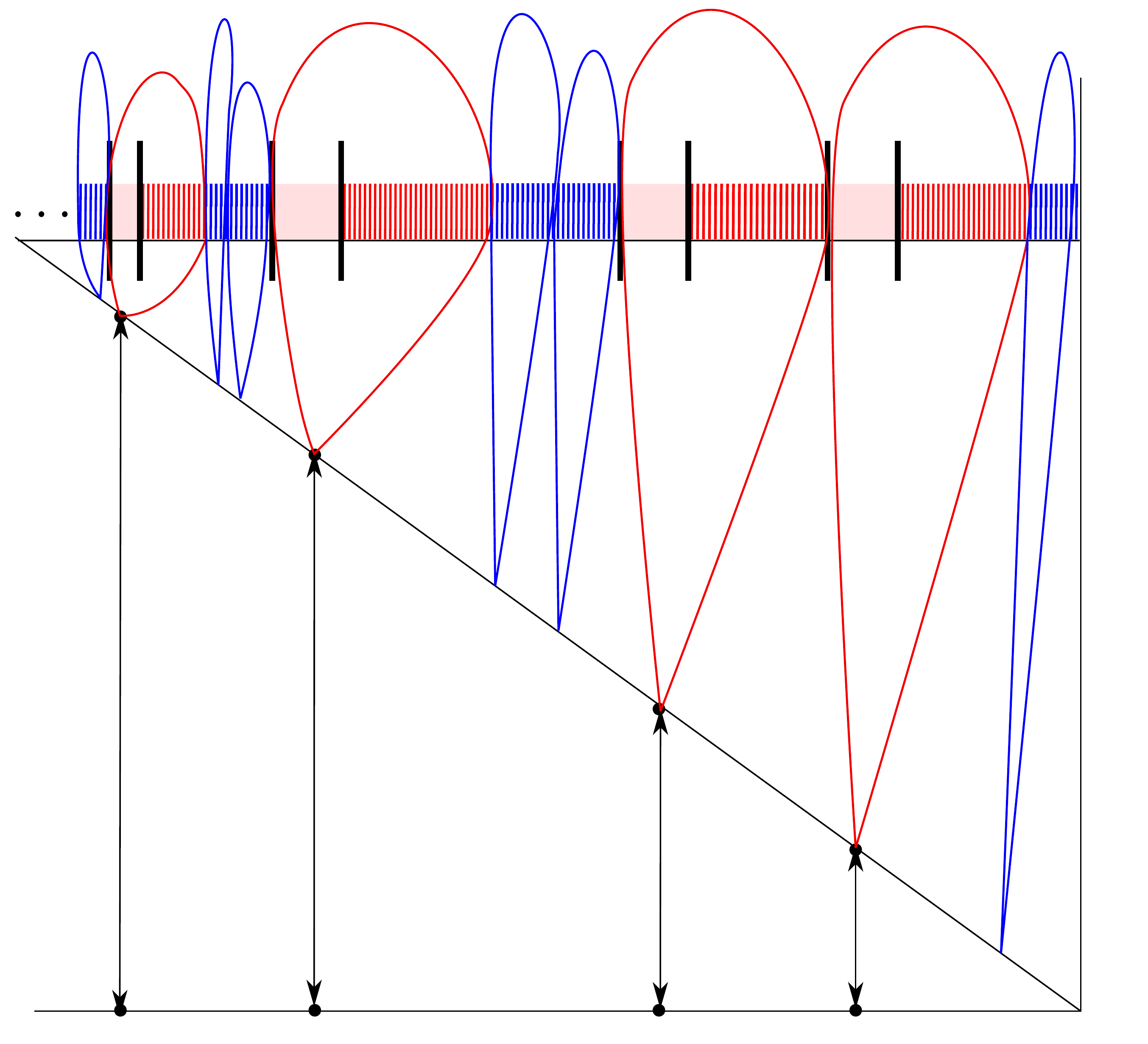}}
\put(208,150){$y$}
\put(27,167){$\gamma^y_4$}
\put(40,167){$\mu^y_4$}
\put(73,167){$\gamma^y_3$}
\put(100,167){$\mu^y_3$}
\put(133,167){$\gamma^y_2$}
\put(145,180){$\mu^y_2=\emptyset$}
\put(178,167){$\gamma^y_1$}
\put(195,167){$\mu^y_1$}
\put(150,135){$(0, H_1^y)$}
\put(112,135){$(0, H_2^y)$}
\put(47,135){$(0, H_3^y)$}
\put(152,20){$s_1^y$}
\put(112,33){$s_2^y$}
\put(48,45){$s_3^y$}
\end{picture}%\vspace{-0.3cm}
	\caption{We illustrate the value of an $\mathrm{SSIP}^{(\alpha)}(\theta_1)$ at level $y>0$. 
		The contribution of each surviving red clade contains a leftmost block of mass $H_i^y$ and the remaining (red-shaded) part $\gamma_i^y$. To the right of each surviving red clade, there is a finite number (possibly zero) of blue clades that form $\mu_i^y$. Non-surviving clades are omitted.}
	\label{fig:MarkClades}
\end{figure}
Given $(s_k^y,\, k\ge 1)$, these $H_k^y$, $\gamma_k^y$, $\mu_k^y$, $k\ge 1$, are conditionally independent.  
To identify the conditional distribution of $\gamma_k^y\concat \mu_k^y$ given $(s^y_k,\,k\ge 1)$, note that display
\eqref{eq:prm-marginal} yields that
 \[
 \Concat_{(s, N_s) ~\text{points of}~\cev{\fF}^{(\mathrm{b})}\colon\! s\in [s^y_{k}, y]\downarrow} \skewer(y-s, N_s,\xi_{N_s} )\sim\mathtt{Gamma}\big(\alpha, (2 (y-s^y_k))^{-1}\big) \cdot\pdip^{(\alpha)}(\alpha),
 \] 
which coincides with the conditional distribution of $\gamma_k^y$ given in \eqref{eq:conddist}, and this interval partition is conditionally independent of $\mu^y_k$. As a result, given $(s^y_k,\,k\ge 1)$ and writing $s^y_0= 0$, the interval partitions $\gamma^y_k\concat \mu^y_k$, $k\ge 1$, are conditionally independent, and using \eqref{eq:prm-marginal} again, we obtain
 \[
\gamma^y_k\concat \mu^y_k\ed\Concat_{(s, N_s) ~\text{points of}~\cev{\fF}^{(\mathrm{b})}\colon\! s\in [s^y_{k-1}, y]\downarrow} \skewer(y-s, N_s,\xi_{N_s} )\sim\mathtt{Gamma}\big(\alpha, (2 (y-s^y_{k-1}))^{-1}\big) \cdot\pdip^{(\alpha)} (\alpha).\]

Next, define an $\mathrm{SSIP}^{(\alpha)} (\theta_1, \alpha)$-evolution by 
$\widetilde{\beta}^y = \underline{\cev{\beta}}^y \concat \vecc{\beta}^y$, $y\ge 0$, where $(\underline{\cev{\beta}},\,y\ge 0)$ is given by \eqref{eq:theta-0} and $(\vecc{\beta}^y ,y\ge 0)$ is an  $\mathrm{RSSIP}^{(\alpha)}(\alpha)$-evolution starting from $\emptyset$. 
Since \eqref{eq:theta-0} is only left-right-reversing within surviving clades, the distribution of the increasing sequence of immigration levels is again given by
\eqref{eq:immlev}, with $\theta=\theta_1-\alpha$, now based on $\underline{\cev{\mathbf{F}}}$. We denote these by 
$(\widetilde{s}_k^y,\,k\ge 1)$ and also write $\widetilde{s}_0^y:=0$. We further read from \eqref{eq:conddist} and \eqref{eq:theta-0} that 
\begin{equation}\label{eq:tildebeta}
(\underline{\cev{\beta}}^y ,\, \vecc{\beta}^y)\ed \left( \Concat_{k=\infty}^1 \widetilde{\gamma}_k^y \concat \{(0,\widetilde{H}_k^y)\} ,\, \widetilde{\gamma}_0^y\right) ,
\end{equation}
%$\widetilde{\beta}^y$ has the form
%\begin{equation}\label{eq:tildebeta}
%\widetilde{\beta}^y\ed \left( \Concat_{k=\infty}^1 \widetilde{\gamma}_k^y \concat \{(0,\widetilde{H}_k^y)\} \right) \concat \widetilde{\gamma}_0^y,
%\end{equation}
where, given $(\widetilde{s}_k^y,\,k\ge 1)$, we have conditionally independent $\widetilde{H}^y_k\sim \mathtt{Gamma}(1-\alpha, (2 (y-\widetilde{s}^y_k))^{-1})$, $k\ge 1$, and 
$\widetilde{\gamma}_k^y\sim \mathtt{Gamma}(\alpha, (2 (y-\widetilde{s}^y_k))^{-1}) \cdot\pdip^{(\alpha)}( \alpha)$, $k\ge 0$. See Figure~\ref{fig:NewConstruction}. 
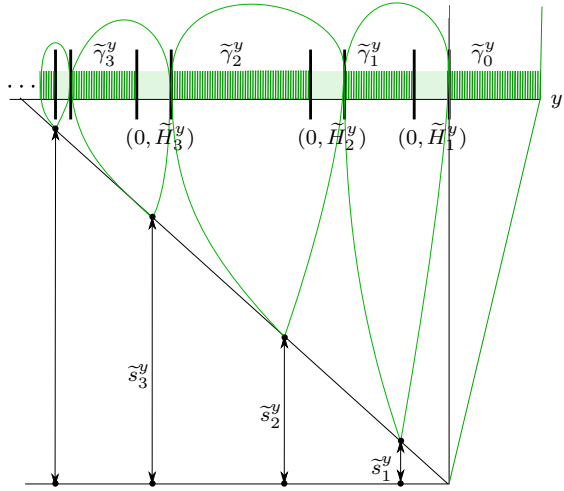
\begin{figure}[t]
	\centering
\begin{picture}(300,190)
		 \put(0,0){\includegraphics[width=0.45\linewidth]{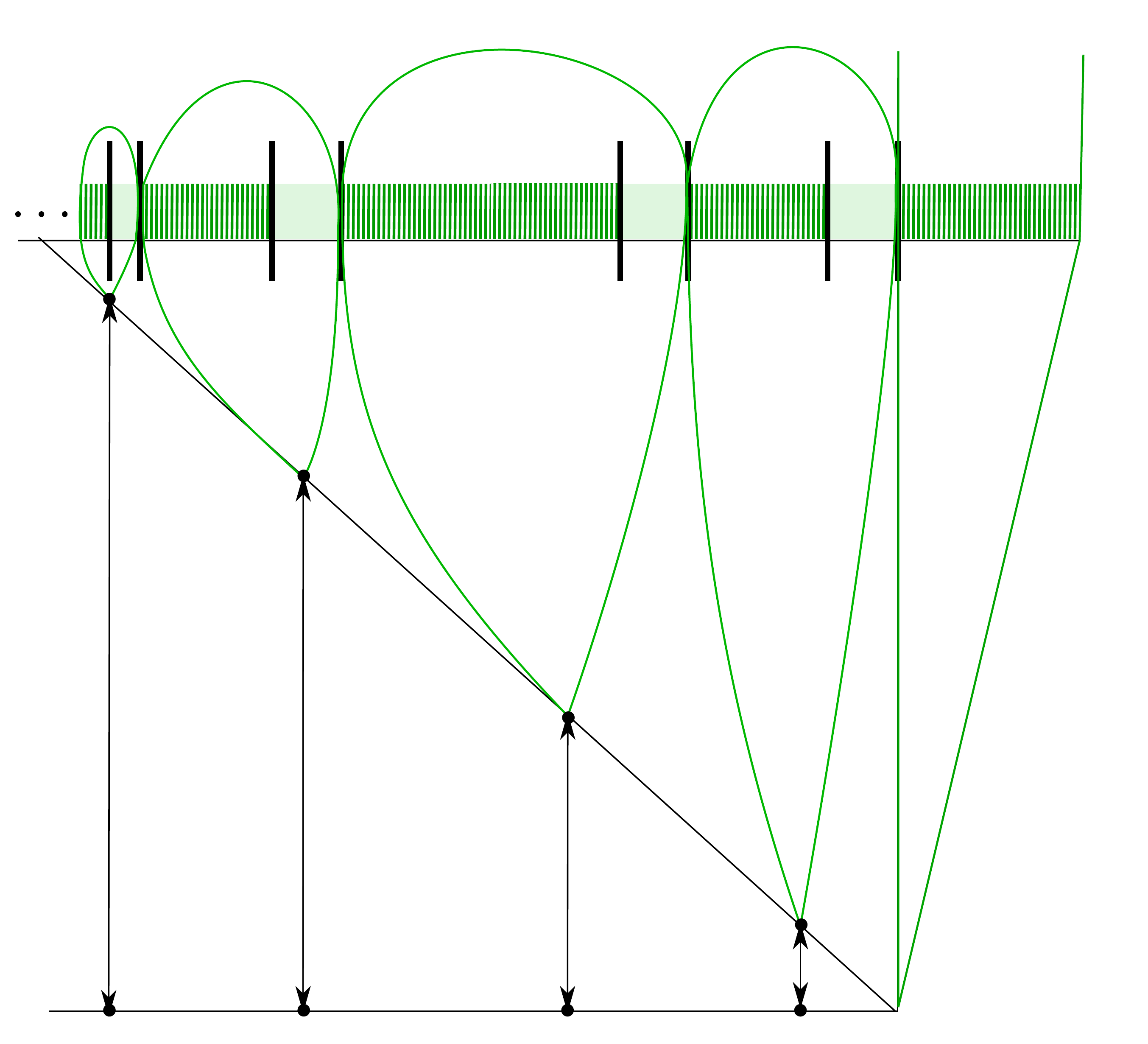}}
\put(208,150){$y$}
\put(35,167){$\widetilde{\gamma}^y_3$}
\put(83,167){$\widetilde{\gamma}^y_2$}
\put(135,167){$\widetilde{\gamma}^y_1$}
\put(178,167){$\widetilde{\gamma}^y_0$}
\put(150,135){$(0, \widetilde{H}_1^y)$}
\put(112,135){$(0, \widetilde{H}_2^y)$}
\put(47,135){$(0, \widetilde{H}_3^y)$}
\put(140,10){$\widetilde{s}_1^y$}
\put(98,30){$\widetilde{s}_2^y$}
\put(48,45){$\widetilde{s}_3^y$}
\end{picture}%\vspace{-0.3cm}
	\caption{The value of an $\mathrm{SSIP}^{(\alpha)}(\theta_1,\alpha)$ at level $y> 0$. 
		The surviving clades of $\protect\cev{\beta}_{\theta_1-\alpha}$ to level $y$ are coloured in green, with each one composed by a rightmost interval and the remaining part (green-shaded).}
	\label{fig:NewConstruction}
\end{figure}

Summarizing, we have the following statement. 

\begin{lemma}\label{lem:mark-correspondence}
	Fix any $y\ge 0$. With notation as above, we have $(s^y_k,\,k\ge 1) \ed (\widetilde{s}^y_k,\,k\ge 1)$.
	Moreover, the conditional distribution of $\big((H_k^y,  \gamma_k^y \concat \mu_k^y),\,k\ge 1\big)$ given $(s^y_k,\,k\ge 1)$ is the same as 
	the conditional distribution of $\big((\widetilde{H}_k^y,  \widetilde{\gamma}_{k-1}^y),\,k\ge 1\big)$ given $(\widetilde{s}^y_k,\,k\ge 1)$. 
\end{lemma}

%\begin{lemma}\label{lem:1-ge}
%	An $\mathrm{SSIP}^{(\alpha)}(\theta_1, \alpha)$-evolution $(\widetilde{\beta}^y, y\ge 0)$ starting from $\gamma$
%	is an $\mathrm{SSIP}^{(\alpha)}(\theta_1)$-evolution. 
%\end{lemma}
\begin{proof}[Proof of Proposition \ref{prop:1-ge}]
	For fixed $y\ge 0$, using Lemma~\ref{lem:mark-correspondence} and its notation, we have the identity in law 
	\begin{equation*}
	\left(\Concat_{i=k}^1 \Big( \{(0,H_i^y)\} \concat \gamma_i^y \concat \mu_i^y \Big) \right) \concat \hat{\beta}^y
	\ed \{(0,\widetilde{H}_k^y)\} \concat \left( \Concat_{i={k-1}}^1 \Big( \widetilde{\gamma}_i^y \concat \{(0,\widetilde{H}_i^y)\} \Big) \right) \concat \widetilde{\gamma}_0^y \concat \hat{\beta}^y, \qquad  k\ge 1, 
	\end{equation*}
where $(\hat{\beta}^z,z\ge 0)$ is an $\mathrm{SSIP}^{(\alpha)}(0)$-evolution starting from $\gamma$, independent of everything else. 
In the limit $k\rightarrow\infty$, the LHS has the law at time $y$ of an $\mathrm{SSIP}^{(\alpha)}(\theta_1)$-evolution starting from $\gamma$. 
For the RHS, since it follows from Corollary~\ref{cor:sym} that $\widetilde{\gamma}_0^y \concat \hat{\beta}^y$ has the law of a $\mathrm{RSSIP}^{(\alpha)}(\alpha)$-evolution at time $y$ starting from $\gamma$, the RHS has, in the limit $k\rightarrow\infty$, the law of an $\mathrm{SSIP}^{(\alpha)}(\theta_1, \alpha)$-evolution at time $y$, starting from $\gamma$. 
So we have identified the one-dimensional marginals. 
It follows from this observation above and the Markov properties of both processes that they have the same finite-dimensional distributions. The claim follows from the path-continuity. 
\end{proof}

As a consequence of Proposition \ref{prop:1-ge}, we can address the apparent lack of left-right-symmetry in Definition \ref{defn:theta-ge}.

\begin{proposition}
	For $\theta_1,\theta_2\ge \alpha$ and $\gamma\in\mathcal{I}_H$, consider three independent processes, an 
	${\rm SSIP}^{(\alpha)}(\alpha)$-evolution $(\widehat{\beta}^y\,y\ge 0)$ starting from $\gamma$, and 
	${\rm SSIP}^{(\alpha)}(\theta_j,0)$-evolutions $(\beta_j^y\,y\ge 0)$, $j=1,2$, starting from $\emptyset$. Then 
	$\beta^y:=\beta_1^y\concat\widehat{\beta}^y\concat{\rm rev}(\beta_2^y)$, $y\ge 0$ is an 
	$\mathrm{SSIP}^{(\alpha)}(\theta_1, \theta_2)$-evolution starting from $\gamma$. 
		
	In particular, the left-right reversal of an $\mathrm{SSIP}^{(\alpha)}(\theta_1,\theta_2)$-evolution is an $\mathrm{SSIP}^{(\alpha)}(\theta_2, \theta_1)$-evolution 
	starting from the left-right-reversed initial state. 
\end{proposition}
\begin{proof}
	First note that $(\beta_1^y,\,y\ge 0)$ has the same distribution as $(\underline{\cev{\beta}}^y,\,y\ge 0)$ in \eqref{eq:theta-0}, by 
	Definition \ref{defn:theta-ge}. Hence, we need to show that $(\widehat{\beta}^y\concat{\rm rev}(\beta_2^y)\,y\ge 0)$ is an
	${\rm RSSIP}^{(\alpha)}(\theta_2)$-evolution starting from $\gamma$, which we defined as the left-right-reversal of 
	${\rm SSIP}^{(\alpha)}(\theta_2)$ starting from ${\rm rev}(\gamma)$. By Proposition \ref{prop:1-ge}, it suffices to show that
	$(\beta_2^y\concat{\rm rev}(\widehat{\beta}^y),\,y\ge 0)$ is an ${\rm SSIP}^{(\alpha)}(\theta_2,\alpha)$-evolution starting from 
	${\rm rev}(\gamma)$. This follows straight from Definition \ref{defn:theta-ge} and Proposition \ref{prop:sym}.
	The final claim follows from the representation in the first part and Proposition \ref{prop:sym} since 
	${\rm rev}(\beta^y)=\beta_2^y\concat{\rm rev}(\widehat{\beta}^y)\concat{\rm rev}(\beta_1^y)$.
\end{proof}

We end this section by deriving two decompositions of Poisson--Dirichlet interval partitions from the correspondence in Lemma~\ref{lem:mark-correspondence}. They have a similar flavour to \cite[Corollary~8]{PitmWink09}, but are different. 
\begin{corollary}
	For $\theta> \alpha$ and $\rho>0$, let $(\bar\beta_n)_{n\ge 1}$ be i.i.d.\@ $\mathtt{PDIP}^{(\alpha)}( 0)$, $(B_n)_{n\ge 1}$ be   i.i.d.\@ $\mathtt{Beta}(\theta- \alpha, 1)$, $(E_n)_{n\ge 1}$ be   i.i.d.\@ $\mathtt{Exp}(\rho)$, and $\gamma\sim \mathtt{Gamma} (\alpha, \rho)\cdot\pdip^{(\alpha)} (\alpha)$. 
	Then we have the identity
	\begin{equation}\label{eq:cor1-mark}
	\left( \Concat_{n=\infty}^1 \Big(E_n \prod_{i=1}^n B_i\Big) \odotip \mathrm{rev}(\bar\beta_n) \right) \concat \gamma  
	\ed \mathtt{Gamma} (\theta, \rho)\cdot\pdip^{(\alpha)}( \theta).
	\end{equation}
\end{corollary}
\begin{proof}
	Consider 	$\widetilde{\beta}^{y}\ed\Concat_{k={\infty}}^1 \left( \widetilde{\gamma}_k^y \concat \{(0,\widetilde{H}_k^y)\} \right)\concat \widetilde{\gamma}_0^y$ given in \eqref{eq:tildebeta}, i.e.\ a decomposition of an $\mathrm{SSIP}^{(\alpha)}(\theta, \alpha)$-evolution at level $y$. 
	With $y= 1/2 \rho$, by \eqref{eq:ssip-clades} and \eqref{eq:pdip:0-alpha} we have
	\[
	\left(\widetilde{\gamma}_n^{1/2 \rho} \concat \left\{(0,\widetilde{H}_n^{1/2 \rho})\right\} , \ \  n\ge 1 \right)
	\ed \left(\left(E_n \prod_{i=1}^n B_i\right) \odotip \mathrm{rev}(\bar\beta_n), \ \  n\ge 1 \right). 
	\]
	Then $\widetilde{\beta}^{1/2 \rho}$ can be written as the LHS of 
	\eqref{eq:cor1-mark}. 
	On the other hand, $\widetilde{\beta}^{1/2 \rho}\sim \mathtt{Gamma} (\theta, \rho)\cdot\pdip^{(\alpha)}( \theta)$ by Proposition~\ref{prop:1-ge} and \eqref{eq:prm-marginal}. 
\end{proof}

\begin{corollary}
	For $\theta>\alpha$ and $\rho>0$, let $(\bar\beta_n)_{n\ge 1}$ be i.i.d.\@ $\mathtt{PDIP}^{(\alpha)}( 0)$, $(B_n)_{n\ge 1}$  i.i.d.\@ $\mathtt{Beta}(\theta, 1)$, $(E_n)_{n\ge 1}$  i.i.d.\@ $\mathtt{Exp}(\rho)$, $G\sim \mathtt{Gamma}(\alpha,\rho)$, $\bar\gamma\sim \pdip^{(\alpha)} (\alpha)$, and $K$ have geometric distribution on $\bN$ with success probability $1-\alpha/\theta$. 
	Then we have the identity
	\[
	\Big(G \prod_{i=1}^K B_i\Big) \odotip \bar\gamma  \concat \left( \Concat_{n=K-1}^1 \Big(E_n \prod_{i=1}^n B_i\Big) \odotip \mathrm{rev}(\bar\beta_n) \right)
	\ed \mathtt{Gamma} (\alpha, \rho)\cdot\pdip^{(\alpha)} (\alpha).
	\]
\end{corollary}
\begin{proof}
	Using the decomposition of $\cev\beta^{1/2 \rho}$ given in \eqref{eq:gammamudecomp} and Lemma~\ref{lem:mark-correspondence}, and notation therein, we look at the interval partition 
	to the right of the rightmost red interval, i.e.\@ $\gamma_{1}^{1/2 \rho}\concat \mu_{1}^{1/2 \rho}$. 
	The Poisson property shows that the first (from the right) red clade is the $K$-th one among all  clades of $\cev{\fF}^{(\mathrm{r})}+ \cev{\fF}^{(\mathrm{b})}$ surviving to level $1/2 \rho$. 
	Using \eqref{eq:ssip-clades} and \eqref{eq:pdip:0-alpha}, we have the representation of the LHS. By Lemma~\ref{lem:mark-correspondence}, $\gamma_{1}^{1/2 \rho}\concat \mu_{1}^{1/2 \rho}$ has the same law as 
	$\widetilde{\gamma}_0^{1/2\rho}\sim\mathtt{Gamma} (\alpha, \rho)\cdot\pdip^{(\alpha)} (\alpha)$.  
\end{proof}

\subsection{Pseudo-stationarity of ${\rm SSIP}^{(\alpha)}(\theta_1,\theta_2)$-evolutions, and the proof of Theorem \ref{thm:pseudo}.}

Recall from the introduction that we call a distribution $\mu$ on $\mathcal{I}_H$ \em pseudo-stationary \em for an interval partition evolution if starting the evolution from an independently scaled multiple of a $\mu$-distributed interval partition, the marginal distributions at all positive times
have the same form. In other words, the evolution only changes the distribution of the total mass, but keeps the distribution of the interval partition normalised to unit total mass invariant. Let us first study ${\rm SSIP}^{(\alpha)}(\theta_1,\theta_2)$-evolutions starting from 
$\emptyset$. 

\begin{proposition}\label{prop:largetheta-empty}
	For $\theta_1\ge \alpha$ and $\theta_2\ge 0$, 
	let $(\beta^y, y\ge 0)$ be an $\mathrm{SSIP}^{(\alpha)}(\theta_1, \theta_2)$-evolution starting from $\emptyset$. 
	%		Then at any $y\ge 0$ we have
	%		\[
	%		\beta^y \sim \mathtt{Gamma}(\theta_1+\theta_2 -\alpha, \frac{1}{2 y}) \odotip \mathtt{PDIP}^{(\alpha)}( \theta_1, \theta_2). 
	%		\]
	Then at any fixed level $y\ge 0$ we have
	\[
	\beta^y \ed  V G_1^y \bar\beta_1 \concat \{(0, V G_0^y)\}\concat G_2^y\bar\beta_2 ,
	\]
	where  $(V,\bar\beta_1,\bar\beta_2, G_1^y,G_0^y, G_2^y)$ is a family of independent random variables with $V\sim \mathtt{Beta}(\theta_1-\alpha, 1)$, $\bar\beta_1\sim \mathtt{PDIP}^{(\alpha)}( \theta_1)$, $\mathrm{rev}(\bar\beta_2)\sim \mathtt{PDIP}^{(\alpha)}( \theta_2)$, $G_1^y \sim \mathtt{Gamma}(\theta_1, 1/2 y)$, $G_0^y \sim \mathtt{Gamma}(1-\alpha, 1/2 y)$, and $G_2^y\sim \mathtt{Gamma}(\theta_2, 1/2 y)$. 
	By convention $V=0$ when $\theta_1=\alpha$ and $G_2^y=0$ when $\theta_2=0$. 	
	In other words, 
	\[
	\beta^y \ed  G_3^y \odotip \Big( V' \odotip \bar\beta_1  \concat \{(0, 1- V')\}\Big)\concat G_2^y\bar\beta_2 ,
	\]
	where $G_3^y \sim \mathtt{Gamma}(\theta_1-\alpha, 1/2 y) $ and $V' \sim \mathtt{Beta} (\theta_1, 1-\alpha)$ are independent, further independent of $(\bar\beta_1, \bar\beta_{2}, G_2^y)$. 
\end{proposition}
%TODO
\begin{proof} 
	We write $\beta^y= \underline{\cev{\beta}}^y \concat \vecc{\beta}^y, y\ge 0$ as in Definition \ref{defn:theta-ge}. It is known from \cite[Proposition~3.6]{IPPAT} that $ \mathrm{rev}(\vecc{\beta}^y)\sim \mathtt{Gamma}(\theta_2, 1/2 y)\cdot \mathtt{PDIP}^{(\alpha)}( \theta_2)$, i.e.\ $\vecc{\beta}^y\ed  G_2^y\bar\beta_2$. 
	
	Using the two-colour correspondence described in Lemma~\ref{lem:mark-correspondence} and its notation, we have 
	$\underline{\cev{\beta}}^y \ed \gamma^y \concat \{(0, A^y)\}$, 
	where $\gamma^y := \Concat_{i=\infty}^2 \left(\{(0,H_i^y)\} \concat \gamma_i^y\concat \mu_i^y \right)$ is the concatenation of all interval partitions to the left of the rightmost red clade, and $A^y:= H_1^y$ is the mass of the leftmost block of the rightmost red clade. 
	Let us enumerate all clades in $\cev{\fF}^{(\mathrm{r})}+\cev{\fF}^{(\mathrm{b})}$ surviving to level $y$ from right to left and denote by $K$ the index of the first red clade (i.e.\@ $\mu_1^y$ is the concatenation of contributions from $K-1$ blue clades). Then $K$ clearly has a geometric distribution with success probability $1-\alpha/\theta_1$. 

	%Iterating \cite[Equation~(3.5)]{IPPAT} and using \cite[Lemma~5.3]{FVAT}, 
%	 we know that for every $k\ge 1$, we can decompose the value of an $\mathrm{SSIP}^{(\alpha)}(\theta_1)$ at level $y$ into 
%	the $k$ rightmost clades and everything to the left them: i.e.\ the value is 
%	$G^y_0 \prod_{i=1}^k B_i \bar\gamma_0 \concat \left(\Concat_{j=k}^1 E^y_j \prod_{i=1}^j B_i \bar{\gamma}_j \right)$, 
%	where $G^y_0\sim \mathtt{Gamma}(\theta_1, \frac{1}{2 y})$, $\bar\gamma_0\sim \mathtt{PDIP}^{(\alpha)}(\theta_1)$,  and for each $j$, $E_j\sim \mathtt{Exponential}(\frac{1}{2y})$, $B_j\sim \mathtt{Beta}(\theta_1, 1)$, $\bar{\gamma}_j\sim \mathtt{PDIP}^{(\alpha)}(0)$; they are all independent. 	
	Applying \eqref{eq:ssip-clades} and \eqref{eq:ssip-clades-bis} to the $\mathtt{SSIP}^{(\alpha)}(\theta_1)$-evolution associated with $\cev{\fF}^{(\mathrm{r})}+\cev{\fF}^{(\mathrm{b})}$ and using \eqref{eq:pdip:0-alpha}, we deduce that,  conditionally on $\{K=k\}$, we have
	$(\gamma^y, A^y)\ed \left( G^y_1 \big(\prod_{i=1}^k B_i \big)\odotip \bar\beta_1,~  G^y_0 \prod_{i=1}^k B_i\right)$, where $G^y_1\sim \mathtt{Gamma}(\theta_1, 1/2 y)$, $G^y_0\sim \mathtt{Gamma}(1-\alpha, 1/2 y)$, $\bar\beta_1\sim \mathtt{PDIP}^{(\alpha)}( \theta_1)$, and $(B_i)_{i\ge 1}$ is an i.i.d.\@ sequence of $\mathtt{Beta}(\theta_1, 1)$; they are all independent. 
	So we complete the proof of the first statement by checking that $V:=\prod_{i=1}^K B_i \sim \mathtt{Beta}(\theta_1-\alpha, 1)$, which follows from the calculation of moments: 
	for every $r\in \bN$, we have 
	\[
	\bE \left[\left(\prod_{i=1}^K B_i\right)^r\right] 
	= \sum_{k=1}^{\infty} \left(\frac{\theta_1}{\theta_1+r} \right)^k \frac{\theta_1-\alpha}{\theta_1}\left(\frac{\alpha}{\theta_1}\right)^k
	= \frac{\theta_1-\alpha}{\theta_1 -\alpha +r}.
	\] 	
	Since $(G_1^y V, G_0^y V)\ed (V' (G_1^y+G_0^y) V, (1-V')(G_1^y+G_0^y) V )\ed (V'G_3^y, (1-V')G_3^y )$, the second statement follows from the first one. 
\end{proof}

\begin{corollary}
		For $\theta\ge \alpha$ and $\rho>0$, consider independent $G_1\sim \mathtt{Gamma} (\theta -\alpha, \rho)$, $B\sim \mathtt{Beta} (\theta, 1-\alpha)$, $G_2\sim \mathtt{Gamma} (\alpha, \rho)$, $\bar{\gamma}_1\sim \pdip^{(\alpha)}( \theta)$,  and $\bar{\gamma}_2\sim  \pdip^{(\alpha)} (\alpha)$,  
	Then we have 
	\[
	G_1 \big(B \odotip \bar{\gamma}_1  \concat \{(0, 1- B)\}\big) \concat G_2  \bar{\gamma}_2  
	\ed \mathtt{Gamma} (\theta, \rho)\cdot\pdip^{(\alpha)}( \theta).
	\]
\end{corollary}
\begin{proof}
     This follows from the marginals of Proposition~\ref{prop:largetheta-empty} with $\theta_1=\theta$, $\theta_2=\alpha$ and $y =1/2\rho$, 
     and from the marginals of an ${\rm SSIP}^{(\alpha)}(\theta)$-evolution recalled in \eqref{eq:prm-marginal}, noting that they must be
     equal by Proposition \ref{prop:1-ge}.
\end{proof}

By using very similar arguments as in \cite[proof of Proposition~3.15, Theorem 1.4(iv)]{IPPAT}, respectively, we deduce the following two consequences of Proposition~\ref{prop:largetheta-empty}.

\begin{lemma}\label{prop:largetheta-gamma}
	For $\theta_1\ge \alpha$, $\theta_2\ge 0$ and $\rho>0$, let $(V,\bar\gamma_1, G_1,\bar\gamma_2, G_2)$ be an independent quintuple with $V\sim \mathtt{Beta}(\theta_1, 1-\alpha)$, $\bar\gamma_1\sim \mathtt{PDIP}^{(\alpha)}( \theta_1)$, $G_1\sim \mathtt{Gamma}(\theta_1 -\alpha, \rho)$, $G_2\sim \mathtt{Gamma}(\theta_2, \rho)$ and $\mathrm{rev}(\bar\gamma_2) \sim \mathtt{PDIP}^{(\alpha)}( \theta_2)$. 
	Let $(\beta^y, y\ge 0)$ be an $\mathrm{SSIP}^{(\alpha)}(\theta_1, \theta_2)$-evolution starting from 
	\[
	\gamma:= \Big( G_1 \odotip \big( V \odotip \bar\gamma_1 \concat \{(0,1-V)\}\big) \Big)\concat G_2 \bar\gamma_2.
	\] 
	Then at any $y\ge 0$, the interval partition $\beta^y$ has the same distribution as 
	$(2 y \rho +1)\odotip \gamma$. 
\end{lemma}

\pagebreak[2]

\begin{proposition}\label{prop:pseudo}
	For $\theta_1\ge \alpha$, $\theta_2\ge 0$ with $\theta:= \theta_1+\theta_2-\alpha$, let $(Z, B, B^\prime,\bar\gamma_1,\bar\gamma_2)$ be an independent quintuple with $Z\sim \besq (2 \theta ) $, $B\sim{\tt Beta}(\theta_1-\alpha,\theta_2)$, $B^\prime\sim{\tt Beta}(1-\alpha,\theta_1)$, $\bar\gamma_1\sim \mathtt{PDIP}^{(\alpha)}( \theta_1)$, and $\mathrm{rev}(\bar\gamma_2)\sim  \mathtt{PDIP}^{(\alpha)}( \theta_2)$. 
	
	Let $(\beta^y, y\ge 0)$ be an $\mathrm{SSIP}^{(\alpha)}(\theta_1, \theta_2)$-evolution starting from 
	$Z(0)\odotip \gamma$, where 
	\[
	\gamma:= B(1-B^\prime) \odotip \bar\gamma_1 \concat \{(0,BB^\prime)\} \concat (1-B) \odotip \bar\gamma_2.
	\] 
	Then for each $y\ge 0$, the interval partition $\beta^y$ has the same distribution as $Z(y)\odotip \gamma$. 
\end{proposition}

\begin{proof}[Proof of Theorem \ref{thm:pseudo}] By definition of pseudo-stationarity, Proposition \ref{prop:pseudo} is just a reformulation of 
  Theorem \ref{thm:pseudo}.
\end{proof}

\subsection{Identification of stopped ${\rm SSIP}^{(\alpha)}(\theta_1,\theta_2)$-evolutions as ${\rm SSIP}^{(\alpha)}_\dagger(\theta_1,\theta_2)$-evolutions.}\label{sec:stop}

We finally show that the two approaches to define a three-parameter family of interval partition evolutions with left and right immigration
lead to the same processes (when stopped upon first reaching $\emptyset$).

\begin{proposition}\label{prop:2-ge}
	An $\mathrm{SSIP}^{(\alpha)}(\theta_1,\theta_2)$-evolution starting from $\gamma$ and stopped when first hitting $\emptyset$, 
	is an $\mathrm{SSIP}^{(\alpha)}_\dagger(\theta_1,\theta_2)$-evolution starting from $\gamma$. 
\end{proposition}
\begin{proof} Recall the construction of an ${\rm SSIP}^{(\alpha)}_\dagger(\theta_1,\theta_2)$-evolution $(\beta^y,\,y\ge 0)$ in Definition 
  \ref{def:stopSSIP}. By Proposition \ref{prop:1-ge}, $\gamma_1^{(0)}$ has the same distribution as the process
  \[
  \underline{\cev{\beta}}^y_1\concat\vecc{\beta}^y_1,\qquad y\ge 0,
  \]
  where $(\underline{\cev{\beta}}_1^y,\,y\ge 0)$ is as in \eqref{eq:theta-0} and $(\vecc{\beta}_1^y,\,y\ge 0)$ is an ${\rm SSIP}^{(\alpha)}(\alpha)$-evolution. Then we can write
  \[
  \beta^y=\underline{\cev{\beta}}^y_1\concat\vecc{\beta}^y_+,\quad\mbox{where }
  \vecc{\beta}^y_+=\vecc{\beta}^y_1\concat\{(0,\mathbf{f}^{(0)}(y))\}\concat\gamma_2^{(0)}(y),\quad 0\le y\le\zeta(\mathbf{f}^{(0)}),
  \]
  is an ${\rm RSSIP}^{(\alpha)}(\theta_2)$-evolution, by Lemma \ref{lem:split} and Proposition \ref{prop:concat}. Comparing with 
  Definition \ref{defn:theta-ge}, the process $(\beta^y,\,0\le y\le\zeta(\mathbf{f}^{(0)}))$ can be viewed as an 
  ${\rm SSIP}^{(\alpha)}(\theta_1,\theta_2)$-evolution stopped at the lifetime of the block starting from the middle interval
  $(\|\beta_1^0\|,\|\beta_1^0\|+m^0)$. Because of the Markov properties of both processes, Theorem \ref{thm:hunt} and Proposition
  \ref{prop:Markov-ge}, we can continue using these arguments to complete the proof inductively. 
\end{proof}

\begin{appendix}
\section{Proofs of Lemmas~\ref{lem:consist2} and \ref{lem:dl}}\label{sec:appx}%% if no title is needed, leave empty \section*{}.

\begin{proof}[Proof of Lemma \ref{lem:dl}] Let $((\beta_1^y,m^y,\beta_2^y),\,y\ge 0)$ be a 
  $\mathcal{J}$-valued ${\rm SSIP}_\dagger^{(\alpha)}(\theta_1,\theta_2)$-evolution as defined in Definition \ref{defn:ipe}. 
  Since total mass evolves continuously between and across any finite number of renaissance times, the total mass reaches zero continuously
  on any event $\{T_n=T_\infty<\infty\}$, $n\ge 0$. Hence, it suffices to show that, on the event $\{T_n\uparrow T_\infty<\infty\}$, 
  the total mass tends to zero along the sequence $(T_n,\,n\ge 0)$. 
  
Recall from \eqref{eq:concatenation-clade} the concatenation of clades $\fN_\gamma= \Concat_{U\in \gamma}\fN_U$ and from Definition~\ref{def:skewer} notation $\skewerbar(\fN_{\gamma},\xi_{\fN_{\gamma}})$, which we here abbreviate as $\skewerbar(\fN_\gamma)$. 
 We use the notation of Definition \ref{defn:ipe}, consider $\mathbf{f}^{(0)}\sim{\tt BESQ}_{m^0}(-2\alpha)$ and 
  independent clade constructions 
  \[\gamma_1^{(0)}=\cev{\beta}_1^{(0)}\concat\skewerbar\bigg(\Concat_{U\in\beta_1^0}\mathbf{N}_U^{(0)}\bigg)
    \quad \text{and}\quad
     \gamma_2^{(0)}={\rm rev}\left(\cev{\beta}_2^{(0)}\!\concat\skewerbar\bigg(\Concat_{U\in\mathrm{rev}(\beta_2^0)}\mathbf{N}_U^{(0)}\bigg)\right),
  \]
  in the sense of \eqref{eq:concatenation-clade} and where $\cev{\beta}_i^{(0)}$ is built from point measures $\cev{\mathbf{F}}_i^{(0)}$ of clades as 
  in \eqref{eq:cev-construction}, with intensities $\theta_i$, $i=1,2$. Our strategy is to use these clades, as well as an auxiliary clade
  $\delta(0,\ff^{(0)})+\mathbf{N}^{(0)}_{\rm mid}:=\clade (\ff^{(0)},\mathbf{N})$ associated with $\ff^{(0)}$, to enhance Definition \ref{defn:ipe} and construct from 
  these clades the entire process $((\beta_1^y,m^y,\beta_2^y),\,y\ge 0)$, as well as a process $(\beta_{\rm em}^y,y\ge 0)$ that is an  
  ${\rm SSIP}^{(\alpha)}(\alpha)$-evolution during $[0,T_\infty)$ and, on $\{T_\infty<\infty\}$, proceeds continuously across $T_\infty$, 
  as an ${\rm SSIP}^{(\alpha)}(0)$-evolution. Indeed, while the blocks of $(\beta_1^y,m^y,\beta_2^y)$ can then be thought of as a subset of 
  the blocks of $\gamma^{(0)}_i(y)$, $i=1,2$, and $\mathbf{f}^{(0)}(y)$, we will make sure that $\beta_{\rm em}^y$ will contain precisely
  the remaining blocks (``emigration''), and the corresponding relationship of the associated total mass processes will yield the claimed asymptotics.
  
  Specifically, suppose by induction that we have constructed the processes for the time interval $[0,T_n]$ for some $n\ge 0$ and are given 
  families $(\mathbf{N}_U^{(n)},U\in\beta_i^{T_n})$, $i=1,2$, and point measures $\cev{\mathbf{F}}_i^{(n)}$, $i=1,2$, of clades, as well as another
  clade $\delta(0,\ff^{(n)})+\mathbf{N}^{(n)}_{\rm mid}$. Furthermore, suppose that, conditionally given the history up to level $T_n$, in the 
  sense of \cite[(3.8) and (3.10)]{IPPAT} and as recalled less formally in Section \ref{sec:prel:clades}, these clades and point measures are independent and so that 
   \begin{equation}\label{eq:gamma-n}
\gamma_1^{(n)}=\cev{\beta}_1^{(n)}\!\concat\skewerbar\bigg(\Concat_{U\in\beta_1^{T_n}}\!\mathbf{N}_U^{(n)}\bigg)
     \quad \text{and}\quad 
     \gamma_2^{(n)}={\rm rev}\left(\cev{\beta}_2^{(n)}\concat\skewerbar\bigg(\Concat_{U\in \mathrm{rev}(\beta_2^{T_n})}\mathbf{N}_U^{(n)}\bigg)\right),
  \end{equation}
  and $\mathbf{f}^{(n)}$ have joint conditional distributions given $((\beta_1^y,m^y,\beta_2^y),\,0\le y\le T_n)$ as in Definition \ref{defn:ipe}.
  Then  
  \[T_{n+1}:= T_n + \zeta(\ff^{(n)}),\qquad
		( \beta_1^{y}, m^{y}, \beta_2^{y}) := \left( \gamma^{(n)}_1(y-T_n) ,\ff^{(n)}(y-T_n) ,\gamma^{(n)}_2(y-T_n) \right), \quad T_n\le y<T_{n+1}, 
		\]
  and $( \beta_1^{T_{n+1}}, m^{T_{n+1}}, \beta_2^{T_{n+1}}):= \phi(\beta_1^{T_{n+1}-} \concat \beta_2^{T_{n+1}-})$ extends the construction of the $\mathcal{J}$-valued process to $[0,T_{n+1}]$ as in Definition \ref{defn:ipe}. 
  To proceed with the induction, we note that $\zeta(\ff^{(n)})$ is independent of the other clades, conditionally given the history up to level $T_n$, so we can apply the Markov-like properties at level $\zeta(\ff^{(n)})$, which we recalled from \cite{Paper1-1} and \cite{IPPAT} in Propositions \ref{prop:markovlike1} and \ref{prop:markovlike2}. Specifically, we obtain point measures of spindles that, via \eqref{eq:concatenation-clade}, can be decomposed into clades and then grouped as $(\mathbf{N}_U^{(n+1)},U\in\beta_i^{T_{n+1}})$, $i=1,2$, and we also obtain point measures $\cev{\mathbf{F}}_i^{(n+1)}$, $i=1,2$, of clades, as well as another clade $\delta(0,\ff^{(n+1)})+\mathbf{N}^{(n+1)}_{\rm mid}$ associated
  with the longest interval of length $m^{T_{n+1}}$, all conditionally independent given the history up to level $T_{n+1}$. Inductively, this completes the construction of Definition \ref{defn:ipe} on $[0,T_\infty)$.
  
  Now set $\beta_{\rm em}^0:=0$ and suppose further that we enter the induction step also with a clade $\delta(0,\mathbf{f}^{(n)})+\mathbf{N}^{(n)}_{\rm mid}$ and an independent process
  \begin{equation}\label{eq:betaemmpartial}
    \beta_{\rm em}^y=\sum_{j=0}^{n-1}\skewer\left(y-T_j,\mathbf{N}_{\rm mid}^{(j)},\zeta(\ff^{(j)})+\xi_{\mathbf{N}_{\rm mid}^{(j)}}\right),
    \quad y\ge 0,
  \end{equation}
  that is an ${\rm SSIP}^{(\alpha)}(\alpha)$-evolution on $[0,T_n]$ continued as an ${\rm SSIP}^{(\alpha)}(0)$-evolution on $[T_n,\infty)$.
  By Lemma \ref{lem:split}, the process $\skewer\left(y,\mathbf{N}_{\rm mid}^{(n)},\zeta(\ff^{(n)})+\xi_{\mathbf{N}_{\rm mid}^{(n)}}\right)$,
  $y\ge 0$, evolves as an ${\rm SSIP}^{(\alpha)}(\alpha)$-evolution on $[0,\zeta(\ff^{(n)})]$. By the Markov-like property at level 
  $\zeta(\ff^{(n)})$, it continues as an ${\rm SSIP}^{(\alpha)}(0)$-evolution. Then the strong Markov property 
  \cite[Proposition 3.14]{Paper1-2} of ${\rm SSIP}^{(\alpha)}(\alpha)$-evolutions yields \eqref{eq:betaemmpartial} with $n$ replaced by $n+1$.
  Inductively, the statement holds for all $n\ge 0$, and by Poisson random measure arguments based on Lemma \ref{lem:descent} and Proposition \ref{prop:cladeconstr2}, this also entails the 
  corresponding statement with $n=\infty$, and in particular, the left limit at $T_\infty$ extends this continuously
  to an ${\rm SSIP}^{(\alpha)}(\alpha)$ on $[0,T_\infty]$. 
  
  By construction, the blocks of $(\beta_1^y,m^y,\beta_2^y)$ are all taken from the skewer at level $y$ of clades that were used in the 
  construction of $\gamma_1^{(0)}$ and $\gamma_2^{(0)}$, and from $\mathbf{f}^{(0)}$. Specifically, this holds explicitly for 
  $0\le y<T_1$. For $T_n\le y<T_{n+1}$, $n\ge 1$, we take skewers at level $y-T_n$ of clades above level $T_n$, which were obtained from 
  the original clades by repeatedly applying Markov-like properties at levels $\zeta(\ff^{(j)})$, $0\le j\le n-1$, and these levels add up to $T_n$. We remark that only the order, not the size of blocks, is affected by the reversals in \eqref{eq:gamma-n}. 
  
  This construction captures at each step all clades above the next level for use either in 
  $((\beta_1^y,m^y,\beta_2^y),\,0\le y<T_\infty)$ or, via $\mathbf{N}^{(n)}_{\rm mid}$, $n\ge 0$, for use in 
  $(\beta_{\rm em}^y,\,0\le y<T_\infty)$. In particular, for all $0\le y<T_\infty$,
  \begin{equation}\label{eq:totalmassdiff}
     \|\beta_1^y\|+m^y+\|\beta_2^y\|
     =\left\|\gamma_1^{(0)}(y)\right\|+\left\|\gamma_2^{(0)}(y)\right\|+\left\|\skewer\left(y,\delta(0,\ff^{(0)})+\mathbf{N}^{(0)}_{\rm mid}\right)\right\|-\|\beta_{\rm em}^y\|.
  \end{equation}
  On the other hand, the size of the longest interval of $((\beta_1^y\concat\{(0,m^y)\}\concat\beta_2^y),\,0\le y<T_\infty)$ tends to zero along times $(T_n,\,n\ge 0)$, when on the event 
  $\{T_n\uparrow T_\infty<\infty\}$: indeed, any subsequence of longest intervals of lengths exceeding $\varepsilon>0$
  would contribute lifetimes that are stochastically bounded below by the lifetimes of an independent sequence of 
  ${\tt BESQ}_\varepsilon(-2\alpha)$, and such lifetimes would have an infinite sum almost surely.

  In the clade construction of $\gamma_1^{(0)}$ and $\gamma_2^{(0)}$, the mass evolution of each block is represented by a spindle in a clade. In our
  construction, each spindle that starts strictly below level $T_\infty$ is either used for $((\beta_1^y,m^y,\beta_2^y),\,0\le y<T_\infty)$ or
  for $(\beta_{\rm em}^y,\,y\ge 0)$. But on $\{T_\infty<\infty\}$, each spindle that straddles level $T_\infty$ must have a positive mass at level 
  $T_\infty$, exceeding some $\varepsilon>0$ on an interval around $T_\infty$. Hence it cannot be included in 
  $((\beta_1^y,m^y,\beta_2^y),\,0\le y<T_\infty)$. But then the RHS of \eqref{eq:totalmassdiff} tends to 0 as $y\uparrow T_\infty$, and this 
  completes the proof.
\end{proof}

\begin{proof}[Proof of Lemma \ref{lem:consist2}] Let $\cJ_2:=\cI_H\times(0,\infty)\times\cI_H\times(0,\infty)\times\cI_H$ and 
  $(\beta_0,m_0,\beta_1,m_1,\beta_2)\in\cJ_2$. We want to couple two ${\rm SSIP}^{(\alpha)}_\dagger(\theta_1,\theta_2)$ 
  starting respectively from $(\beta_0,m_0,\beta_1\concat\{(0,m_1)\}\concat\beta_2)$ and 
  $(\beta_0\concat\{(0,m_0)\}\concat\beta_1,m_1,\beta_2)$ such that the associated $\cI_H$-valued processes coincide. 
  
  Viewing $\cJ$-valued processes as $\cI_H$-valued processes split around a marked block, we now construct a $\cJ_2$-valued process that
  captures two marked blocks. Specifically, we construct $((\beta_0^y,m_0^y,\beta_1^y,m_1^y,\beta_2^y),\,0\le y<S_N)$ starting from 
  $(\beta_0,m_0,\beta_1,m_1,\beta_2)$ at time $S_0:=0$ by the following inductive steps, mimicking Definition \ref{defn:ipe}. Suppose
  that we have constructed the process on $[0,S_n]$ for some $n\ge 0$ and some 
  $(\beta_0^{S_n},m_0^{S_n},\beta_1^{S_n},m_1^{S_n},\beta_2^{S_n})\in\mathcal{J}_2$.
  \begin{itemize}
    \item Conditionally on the history, consider, independently, an ${\rm SSIP}^{(\alpha)}(\theta_1)$-evolution 
      $\gamma_0^{(n)}$ starting from $\beta_0^{S_n}$, an ${\rm SSIP}^{(\alpha)}(\alpha)={\rm RSSIP}^{(\alpha)}(\alpha)$-evolution 
      $\gamma_1^{(n)}$ starting from $\beta_1^{S_n}$, an ${\rm RSSIP}^{(\alpha)}(\theta_2)$-evolution $\gamma_2^{(n)}$ starting 
      from $\beta_2^{S_n}$, and $\ff^{(n)}_i\sim{\tt BESQ}_{m_i^{S_n}}(-2\alpha)$, $i=0,1$. Let 
      $\Delta_n:=\min\{\zeta(\ff_0^{(n)}),\zeta(\ff_1^{(n)})\}$ and $S_{n+1}:=S_n+\Delta_n$, and define
      \[\left(\beta_0^{S_n+y},m_0^{S_n+y},\beta_1^{S_n+y},m_1^{S_n+y},\beta_2^{S_n+y}\right):=
         \left(\gamma_0^{(n)}(y),\ff_0^{(n)}(y),\gamma_1^{(n)}(y),\ff_1^{(n)}(y),\gamma_2^{(n)}(y)\right),\quad 0\le y<\Delta(S_n).\]
    \item If $\Delta_n=\zeta(\ff_i^{(n)})$ for some $i=0,1$, and $\ff_{1-i}^{(n)}(\Delta_n)$ exceeds the length of the longest interval in 
      $\gamma_j^{(n)}(\Delta_n)$ for all $j=0,1,2$, let $N=n+1$. The construction is complete.
    \item Otherwise, identify the longest interval and split the associated $\gamma_j^{(n)}(\Delta_n)$ around this interval. This results in a total
      of four interval partitions and two blocks. In the natural order, two of these interval partitions are adjacent. Concatenate these two and 
      collect the now five components as $\big(\beta_0^{S_{n+1}},m_0^{S_{n+1}},\beta_1^{S_{n+1}},m_1^{S_{n+1}},\beta_2^{S_{n+1}}\big)$.
  \end{itemize}
  Note that (in general) we may have $N\in\mathbb{N}\cup\{\infty\}$. On the event $\{N<\infty\}$, we further continue the evolution as a 
  $\cJ$-valued process starting from the terminal value of the $\cJ_2$-valued process, with adjacent interval partitions concatenated. 
  
  By concatenation properties of ${\rm SSIP}^{(\alpha)}(\theta_1)$- and ${\rm RSSIP}^{(\alpha)}(\theta_2)$-evolutions (Proposition \ref{prop:concat} and Lemma \ref{lem:split}) and by the 
  strong Markov property of these processes applied at the stopping times $S_n$, $n\ge 1$, we obtain two coupled ${\rm SSIP}^{(\alpha)}_\dagger(\theta_1,\theta_2)$-evolutions, which induce the
  same $\mathcal{I}_H$-valued process, as required. 
Indeed, the construction of these two processes is clearly complete on $\{N<\infty\}$
  and on $\{N=\infty,S_\infty=\infty\}$ with $S_{\infty}= \lim_{n\to \infty}S_n$. This suffices if the event $\{N=\infty,S_\infty<\infty\}$ has zero probability. Otherwise, on 
  $\{N=\infty,S_\infty<\infty\}$ the construction of at least one process is complete and by Lemma \ref{lem:dl}, the total mass tends to
  zero along a subsequence of $(S_n,\, n\ge 0)$, and hence the other construction cannot remain unfinished with blocks of positive size at 
  $S_\infty$. 
\end{proof}

\end{appendix}

%%%%%%%%%%%%%%%%%%%%%%%%%%%%%%%%%%%%%%%%%%%%%%
%% Single Appendix:                         %%
%%%%%%%%%%%%%%%%%%%%%%%%%%%%%%%%%%%%%%%%%%%%%%
%\begin{appendix}
%\section*{???}%% if no title is needed, leave empty \section*{}.
%\end{appendix}
%%%%%%%%%%%%%%%%%%%%%%%%%%%%%%%%%%%%%%%%%%%%%%
%% Multiple Appendixes:                     %%
%%%%%%%%%%%%%%%%%%%%%%%%%%%%%%%%%%%%%%%%%%%%%%
%\begin{appendix}
%\section{???}
%
%\section{???}
%
%\end{appendix}

%%%%%%%%%%%%%%%%%%%%%%%%%%%%%%%%%%%%%%%%%%%%%%
%% Support information (funding), if any,   %%
%% should be provided in the                %%
%% Acknowledgements section.                %%
%%%%%%%%%%%%%%%%%%%%%%%%%%%%%%%%%%%%%%%%%%%%%%
\section*{Acknowledgements}
QS was partially supported by SNSF grant P2ZHP2\_171955.

%%%%%%%%%%%%%%%%%%%%%%%%%%%%%%%%%%%%%%%%%%%%%%%%%%%%%%%%%%%%%
%%                  The Bibliography                       %%
%%                                                         %%
%%  imsart-number.bst  will be used to                     %%
%%  create a .BBL file for submission.                     %%
%%                                                         %%
%%  Note that the displayed Bibliography will not          %%
%%  necessarily be rendered by Latex exactly as specified  %%
%%  in the online Instructions for Authors.                %%
%%                                                         %%
%%  MR numbers will be added by VTeX.                      %%
%%                                                         %%
%%  Use \cite{...} to cite references in text.             %%
%%                                                         %%
%%%%%%%%%%%%%%%%%%%%%%%%%%%%%%%%%%%%%%%%%%%%%%%%%%%%%%%%%%%%%

%% if your bibliography is in bibtex format, uncomment commands:
%MW\bibliographystyle{imsart-number} % Style BST file
\bibliographystyle{abbrv}
\bibliography{AldousDiffusion4}       % Bibliography file (usually '*.bib')

%% or include bibliography directly:
% \begin{thebibliography}{}
% \bibitem{b1}
% \end{thebibliography}

\end{document}

%% file: Fig_JCCP_skewer_5b.pdf_t
\begin{picture}(0,0)%
\includegraphics[scale=0.96]{Fig_JCCP_skewer_5b.pdf}%
\end{picture}%
\setlength{\unitlength}{3978sp}%was 4144
\begingroup\makeatletter\ifx\SetFigFont\undefined%
\gdef\SetFigFont#1#2#3#4#5{%
  \reset@font\fontsize{#1}{#2pt}%
  \fontfamily{#3}\fontseries{#4}\fontshape{#5}%
  \selectfont}%
\fi\endgroup%
\begin{picture}(5901,2233)(304,-2714)
\put(3646,-1556){\makebox(0,0)[b]{\smash{{\SetFigFont{10}{12.0}{\familydefault}{\mddefault}{\updefault}{\color[rgb]{0,0,0}$t$}%
}}}}
\put(406,-1276){\makebox(0,0)[rb]{\smash{{\SetFigFont{11}{13.2}{\familydefault}{\mddefault}{\updefault}{\color[rgb]{0,0,0}$y$}%
}}}}
\put(4501,-638){\makebox(0,0)[rb]{\smash{{\SetFigFont{10}{12.0}{\familydefault}{\mddefault}{\updefault}{\color[rgb]{0,0,0}$(N,X)$}%
}}}}
\put(4501,-811){\makebox(0,0)[rb]{\smash{{\SetFigFont{10}{12.0}{\familydefault}{\mddefault}{\updefault}{\color[rgb]{0,0,0}or $(V,X)$}%
}}}}
\put(3511,-1096){\makebox(0,0)[rb]{\smash{{\SetFigFont{11}{13.2}{\familydefault}{\mddefault}{\updefault}{\color[rgb]{0,0,0}$X(t)$}%
}}}}
\put(3781,-2086){\makebox(0,0)[lb]{\smash{{\SetFigFont{11}{13.2}{\familydefault}{\mddefault}{\updefault}{\color[rgb]{0,0,0}$X(t-)$}%
}}}}
\put(4775,-2396){\makebox(0,0)[lb]{\smash{{\SetFigFont{10}{12.0}{\familydefault}{\mddefault}{\updefault}{\color[rgb]{0,0,0}$0$}%
}}}}
\put(5393,-632){\makebox(0,0)[b]{\smash{{\SetFigFont{10}{12.0}{\familydefault}{\mddefault}{\updefault}{\color[rgb]{0,0,0}$(f_t(z),\,z\geq 0)$}%
}}}}
\put(5196,-2396){\makebox(0,0)[lb]{\smash{{\SetFigFont{10}{12.0}{\familydefault}{\mddefault}{\updefault}{\color[rgb]{0,0,0}$x_t$}%
}}}}
\put(6100,-2396){\makebox(0,0)[lb]{\smash{{\SetFigFont{10}{12.0}{\familydefault}{\mddefault}{\updefault}{\color[rgb]{0,0,0}$1$}%
}}}}
\put(3144,-2364){\makebox(0,0)[b]{\smash{{\SetFigFont{11}{13.2}{\familydefault}{\mddefault}{\updefault}{\color[rgb]{0,0,0}$f_t(y-X(t-))$}%
}}}}
\put(6043,-1440){\makebox(0,0)[lb]{\smash{{\SetFigFont{10}{12.0}{\familydefault}{\mddefault}{\updefault}{\color[rgb]{0,0,0}$z$}%
}}}}
\put(1807,-2645){\makebox(0,0)[rb]{\smash{{\SetFigFont{11}{13.2}{\familydefault}{\mddefault}{\updefault}{\color[rgb]{0,0,0}$\skewer(y,N,X)$}%
}}}}
\put(6190,-2645){\makebox(0,0)[rb]{\smash{{\SetFigFont{11}{13.2}{\familydefault}{\mddefault}{\updefault}{\color[rgb]{0,0,0}$\ensuremath{\normalfont\textsc{sSkewer}}(y,V,X)$}%
}}}}
\end{picture}%

%% file: emimm-arXiv.bbl
\def\polhk#1{\setbox0=\hbox{#1}{\ooalign{\hidewidth
  \lower1.5ex\hbox{`}\hidewidth\crcr\unhbox0}}}
\begin{thebibliography}{10}

\bibitem{Ald-web}
D.~Aldous.
\newblock {Problem. Give a rigorous construction of this "diffusion on
  continuum trees" http://www.stat.berkeley.edu/$\sim$aldous/Research/OP/
  fw.html}, 1999.

\bibitem{Bec07}
J.~Bect.
\newblock {\em {Processus de Markov diffusifs par morceaux: outils analytiques
  et num\'eriques.}}
\newblock Th\`ese de doctorat, Universit\'e Paris-Sud XI, 171 p., also
  available at https://tel.archives-ouvertes.fr/tel-00169791, 2007.

\bibitem{CFW}
B.~Chen, D.~Ford, and M.~Winkel.
\newblock A new family of {M}arkov branching trees: the alpha-gamma model.
\newblock {\em Electron. J. Probab.}, 14:no. 15, 400--430 (electronic), 2009.

\bibitem{CuriKort2014}
N.~Curien and I.~Kortchemski.
\newblock Random stable looptrees.
\newblock {\em Electron. J. Probab.}, 19:35 pp., 2014.

\bibitem{DuplMillShef14}
B.~Duplantier, J.~Miller, and S.~Sheffield.
\newblock Liouville quantum gravity as a mating of trees.
\newblock arXiv:1409.7055 [math.PR], 2014.

\bibitem{DuquLeGall02}
T.~Duquesne and J.-F. Le~Gall.
\newblock Random trees, {L}\'evy processes and spatial branching processes.
\newblock {\em Ast\'erisque}, (281):vi+147, 2002.

\bibitem{Ethier14}
S.~Ethier.
\newblock A property of {P}etrov's diffusion.
\newblock {\em Electron. Commun. Probab.}, 19:no. 65, 1--4, 2014.

\bibitem{FengSun10}
S.~Feng and W.~Sun.
\newblock Some diffusion processes associated with two parameter
  {P}oisson--{D}irichlet distribution and {D}irichlet process.
\newblock {\em Probab. Theory Related Fields}, 148(3-4):501--525, 2010.

\bibitem{FengSun19}
S.~Feng and W.~Sun.
\newblock A dynamic model for the two-parameter {D}irichlet process.
\newblock {\em Potential Analysis}, 51(2):147--164, 2019.

\bibitem{Paper4}
N.~Forman, S.~Pal, D.~Rizzolo, and M.~Winkel.
\newblock {Aldous diffusion I: a projective system of continuum $k$-tree
  evolutions}.
\newblock arXiv:1809.07756 [math.PR], 2018.

\bibitem{Paper0}
N.~Forman, S.~Pal, D.~Rizzolo, and M.~Winkel.
\newblock Uniform control of local times of spectrally positive stable
  processes.
\newblock {\em Ann. Appl. Probab.}, 28(4):2592--2634, 2018.

\bibitem{Paper1-1}
N.~Forman, S.~Pal, D.~Rizzolo, and M.~Winkel.
\newblock {Diffusions on a space of interval partitions: construction from
  marked L{\'e}vy processes}.
\newblock {\em Electron. J. Probab.}, 25:46 pp., 2020.

\bibitem{Paper1-2}
N.~Forman, S.~Pal, D.~Rizzolo, and M.~Winkel.
\newblock Diffusions on a space of interval partitions: {Poisson--Dirichlet}
  stationary distributions.
\newblock {\em to appear in Ann. Probab.}, 2020+.
\newblock preprint available as arXiv:1910.07626 [math.PR].

\bibitem{Paper1-3}
N.~Forman, S.~Pal, D.~Rizzolo, and M.~Winkel.
\newblock Interval partition diffusions: connection with {P}etrov's
  {Poisson--Dirichlet} diffusions.
\newblock Work in progress, 2020.

\bibitem{Paper1-0}
N.~Forman, S.~Pal, D.~Rizzolo, and M.~Winkel.
\newblock Metrics on sets of interval partitions with diversity.
\newblock {\em Electron. Commun. Probab.}, 25:16 pp., 2020.

\bibitem{FVAT}
N.~Forman, D.~Rizzolo, Q.~Shi, and M.~Winkel.
\newblock {A two-parameter family of measure-valued diffusions with
  Poisson--Dirichlet stationary distributions}.
\newblock arXiv:2007.05250 [math.PR], 2020.

\bibitem{IPPAT}
N.~Forman, D.~Rizzolo, Q.~Shi, and M.~Winkel.
\newblock {Diffusions on a space of interval partitions: The two-parameter
  model}.
\newblock arXiv:2008.02823 [math.PR], 2020.

\bibitem{GnedPitm05}
A.~Gnedin and J.~Pitman.
\newblock Regenerative composition structures.
\newblock {\em Ann. Probab.}, 33(2):445--479, 2005.

\bibitem{GoinYor03}
A.~G{\"o}ing-Jaeschke and M.~Yor.
\newblock A survey and some generalizations of {B}essel processes.
\newblock {\em Bernoulli}, 9(2):313--349, 2003.

\bibitem{HM04}
B.~Haas and G.~Miermont.
\newblock The genealogy of self-similar fragmentations with negative index as a
  continuum random tree.
\newblock {\em Electron. J. Probab.}, 9:no. 4, 57--97 (electronic), 2004.

\bibitem{HPW}
B.~Haas, J.~Pitman, and M.~Winkel.
\newblock Spinal partitions and invariance under re-rooting of continuum random
  trees.
\newblock {\em Ann. Probab.}, 37(4):1381--1411, 2009.

\bibitem{IshwJame03}
H.~Ishwaran and L.~F. James.
\newblock Generalized weighted chinese restaurant processes for species
  sampling mixture models.
\newblock {\em Statistica Sinica}, pages 1211--1235, 2003.

\bibitem{LeGaMier2011}
J.-F. Le~Gall and G.~Miermont.
\newblock Scaling limits of random planar maps with large faces.
\newblock {\em Ann. Probab.}, 39(1):1--69, 2011.

\bibitem{LohrMytnWint18}
W.~L{\"o}hr, L.~Mytnik, and A.~Winter.
\newblock {The Aldous chain on cladograms in the diffusion limit}.
\newblock {\em Ann. Probab.}, 48(5):2565--2590, 2020.

\bibitem{Mey75}
P.~A. Meyer.
\newblock {Renaissance, recollements, m\'elanges, ralentissement de processus
  de Markov}.
\newblock {\em Ann. Inst. Fourier}, 25(3-4):465--497, 1975.

\bibitem{NussWint20}
J.~Nussbaumer and A.~Winter.
\newblock {The algebraic $\alpha$-Ford tree under evolution}.
\newblock arXiv:2006.09316 [math.PR], 2020.

\bibitem{Petrov09}
L.~A. Petrov.
\newblock A two-parameter family of infinite-dimensional diffusions on the
  {K}ingman simplex.
\newblock {\em Funktsional. Anal. i Prilozhen.}, 43(4):45--66, 2009.

\bibitem{CSP}
J.~Pitman.
\newblock {\em Combinatorial stochastic processes}, volume 1875 of {\em Lecture
  Notes in Mathematics}.
\newblock Springer-Verlag, Berlin, 2006.
\newblock Lectures from the 32nd Summer School on Probability Theory held in
  Saint-Flour, July 7--24, 2002.

\bibitem{PitmWink09}
J.~Pitman and M.~Winkel.
\newblock Regenerative tree growth: binary self-similar continuum random trees
  and {P}oisson--{D}irichlet compositions.
\newblock {\em Ann. Probab.}, 37(5):1999--2041, 2009.

\bibitem{PW18}
J.~Pitman and M.~Winkel.
\newblock {Squared Bessel processes of positive and negative dimension embedded
  in Brownian local times}.
\newblock {\em Electron. Commun. Probab.}, 23:13 pp., 2018.

\bibitem{PitmYor82}
J.~Pitman and M.~Yor.
\newblock A decomposition of {B}essel bridges.
\newblock {\em Z. Wahrsch. Verw. Gebiete}, 59(4):425--457, 1982.

\bibitem{PitmYorPDAT}
J.~Pitman and M.~Yor.
\newblock The two-parameter {P}oisson--{D}irichlet distribution derived from a
  stable subordinator.
\newblock {\em Ann. Probab.}, 25(2):855--900, 1997.

\bibitem{RevuzYor}
D.~Revuz and M.~Yor.
\newblock {\em Continuous martingales and {B}rownian motion}, volume 293 of
  {\em Grundlehren der Mathematischen Wissenschaften [Fundamental Principles of
  Mathematical Sciences]}.
\newblock Springer-Verlag, Berlin, third edition, 1999.

\bibitem{RivRiz}
K.~Rivera-Lopez and D.~Rizzolo.
\newblock Diffusive limits of two-parameter ordered chinese restaurant process
  up-down chains.
\newblock arXiv:2011.06577 [math.PR], 2020.

\bibitem{RogeWink20}
D.~Rogers and M.~Winkel.
\newblock {A Ray--Knight representation of up-down Chinese restaurants}.
\newblock arXiv:2006.06334 [math.PR], 2020.

\bibitem{Ruggiero14}
M.~Ruggiero.
\newblock Species dynamics in the two-parameter {P}oisson--{D}irichlet
  diffusion model.
\newblock {\em J. Appl. Probab.}, 51(1):174--190, 2014.

\bibitem{RuggWalk09}
M.~Ruggiero and S.~G. Walker.
\newblock Countable representation for infinite dimensional diffusions derived
  from the two-parameter {P}oisson--{D}irichlet process.
\newblock {\em Electron. Commun. Probab.}, 14:501--517, 2009.

\bibitem{Schweinsberg02}
J.~Schweinsberg.
\newblock An {$O(n^2)$} bound for the relaxation time of a {M}arkov chain on
  cladograms.
\newblock {\em Random Structures Algorithms}, 20(1):59--70, 2002.

\bibitem{QSMW2}
Q.~Shi and M.~Winkel.
\newblock {Up-down ordered Chinese Restaurant Processes with two-sided
  immigration, diffusion limits and emigration}.
\newblock Work in progress, 2020.

\bibitem{Teh06}
Y.~W. Teh.
\newblock {A hierarchical Bayesian language model based on Pitman--Yor
  processes}.
\newblock In {\em Proceedings of the 21st International Conference on
  Computational Linguistics and the 44th annual meeting of the Association for
  Computational Linguistics}, pages 985--992. Association for Computational
  Linguistics, 2006.

\end{thebibliography}
